\documentclass[10pt]{article}
\usepackage[margin=1in]{geometry}

\input xy
\xyoption{all}

\usepackage{amsmath,amsfonts,amsthm,amssymb,latexsym,amscd,amsopn,graphics,color,enumerate,lscape,enumitem,mathtools,multirow,verbatim}
\usepackage[colorlinks]{hyperref}
\usepackage{xcolor}
\hypersetup{
    colorlinks = false,
    linkbordercolor = {white},}

\usepackage{tikz-cd}
\usepackage{tikz}
\usetikzlibrary{positioning}

\allowdisplaybreaks
\theoremstyle{plain}
  \newtheorem{theorem}{Theorem}
  
   \newtheorem{thm}{Theorem}[section]
  
  \newtheorem{prop}[thm]{Proposition}
  \newtheorem{proposition}[thm]{Proposition}
  \newtheorem{cor}[thm]{Corollary}
  \newtheorem{lemma}[thm]{Lemma}
  
\theoremstyle{definition}
  \newtheorem{defn}[thm]{Definition}

  \newtheorem{remark}[thm]{Remark}
  
\theoremstyle{remark}

\DeclareMathOperator{\Cl}{Cl}

\DeclareMathOperator{\Res}{Res}

\DeclareMathOperator{\rad}{rad}
\DeclareMathOperator{\Hom}{Hom}

\DeclareMathOperator{\Aut}{Aut}

\DeclareMathOperator{\Sym}{Sym}

\renewcommand{\Im}{\frak{Im}}
\renewcommand{\Re}{\frak{Re}}

\DeclareMathOperator{\cusp}{cusp}
\DeclareMathOperator{\Norm}{Nm}

\def\Z{{\mathbb Z}}
\def\T{{\mathcal T}}
\def\CS{{\mathsf S}}

\def\all{{\rm all}}

\def\Gal{{\rm Gal}}

\def\GL{{\rm GL}}

\def\Cl{{\rm Cl}}

\def\Sym{{\rm Sym}}

\def\Aut{{\rm Aut}}

\def\Mass{{\rm Mass}}

\def\Ind{{\rm Ind}}
\def\im{{\rm im}}
\def\coker{{\rm coker}}
\def\un{{\rm un}}

\def\Disc{{\Delta}}

\def\Vol{{\rm Vol}}
\def\Val{{\nu}}
\def\val{{\nu}}

\def\R{{\mathbb R}}

\def\bA{{\mathbb A}}
\def\bR{{\mathbb R}}

\def\bF{{\mathbb F}}
\def\FF{{\mathcal F}}

\def\RR{{\mathcal R}}
\def\I{{\mathcal I}}
\def\Q{{\mathbb Q}}
\def\H{{\rm H}}

\def\Z{{\mathbb Z}}

\def\Q{{\mathbb Q}}
\def\O{{\mathcal O}}

\def\bQ{{\mathbb Q}}
\def\cC{{\mathcal C}}
\def\cD{{\mathcal D}}
\def\cF{{\mathcal F}}

\def\cT{{\mathcal T}}

\def\bF{{\mathbb F}}
\def\bQ{{\mathbb Q}}
\def\cO{{\mathcal O}}

\def\cN{{\mathcal N}}
\def\cS{{\mathcal S}}

\def\fp{{\mathfrak p}}

\newcommand{\rotation}{\sigma}
\newcommand{\reflection}{\tau}

\newcommand{\A}{\mathbb{A}}

\newcommand{\cP}{\mathcal{P}}

\newcommand{\outeraut}{\varphi}

\newcommand{\D}{{\rm d}}

\newcommand{\q}{{\rm q}}

\newcommand{\cA}{{\mathcal A}}
\DeclareMathOperator{\Nm}{Nm}

\usepackage{xcolor}
\definecolor{darkred}{HTML}{CC1F1F}
\definecolor{green}{rgb}{.4,.7,.4}
\definecolor{blue}{rgb}{.2,.6,.75}
\definecolor{pastelyellow}{rgb}{0.992157, 0.552941, 0.235294}
\definecolor{pastelorange}{rgb}{0.941176, 0.231373, 0.12549}
\definecolor{pastelred}{rgb}{0.741176, 0., 0.14902}
\definecolor{darkbrown}{rgb}{0.25098, 0., 0.0745098}
\title{Malle's Conjecture for Galois octic fields over $\bQ$}

\author{Arul Shankar and Ila Varma}

\begin{document}

\maketitle

\begin{abstract}
We compute the asymptotic number of octic number fields whose Galois
groups over $\bQ$ are isomorphic to $D_4$, the symmetries of a square,
when ordering such fields by their absolute discriminants. In
particular, we verify the strong form of Malle's conjecture
\cite{MR2068887} for such octic $D_4$-fields and obtain the constant of
proportionality. Our result answers the question of whether a positive
proportion of Galois octic extensions of $\bQ$ have non-abelian Galois
group in the negative. We further demonstrate that the constant of
proportionality satisfies the Malle--Bhargava principle
\cite{MR2068887,MR2354798} of being a product of local masses, despite
the fact that this principle does {\em not} hold for discriminants of
quartic $D_4$-fields (see \cite{MR1918290,D4preprint}).

This is the first instance of asymptotics being recovered for a
non-concentrated family (in the sense of \cite{ALWWPreprint}) of
number fields of Galois group neither abelian nor
symmetric. Previously, this was only known for abelian fields
\cite{MR969545}, degree-$n$ $S_n$-fields for $n=3,4,5$ (in
\cite{MR491593,MR2183288,MR2745272}), and degree-$6$ $S_3$-fields
\cite{MR2373587}.
\end{abstract}

%\tableofcontents

%\input{intro.tex}
\section{Introduction}

The purpose of this article is to prove the strong form of Malle's
conjecture \cite{MR2068887} on the count of octic number
fields having Galois group $D_4$ and bounded discriminants, by determining the asymptotic growth of this count. Let
$N_8(G,X)$ denote the number of (isomorphism classes of) normal octic
extensions over $\bQ$ with Galois group isomorphic to $G$. Work of
Wright~\cite{MR969545} on counting abelian fields that predated and inspired
Malle's formulation of his conjectures demonstrated
\begin{align*}
N_8(C_8,X) &\sim c_8X^{1/4}, \\
N_8(C_2 \times C_4, X) &\sim c_{2,4}X^{1/4} \log^2 X, \mbox{ and } \\
N_8(C_2 \times C_2 \times C_2, X) &\sim c_{2,2,2}X^{1/4} \log^6 X,
\end{align*}
where $C_n$ denote the cyclic group of order $n$, and $c_8$, $c_{2,4}$, and $c_{2,2,2}$ are explicit constants.
For the quaternion group $Q_8$, Kluners \cite{habilitation} determined
that $dX^{1/4} \leq N_8(Q_8,X) \leq d'X^{1/4}$ for some positive real
constants $d$ and $d'$, and recent work of Koymans--Pagano \cite{KP}determines exact asymptotics for the count. The remaining and most elusive Galois octic
case occurs when $G = D_4$, the dihedral group of order 8, which is
the subject of our first main result.
  
If $M$ denotes a number field of degree 8 that is normal over $\bQ$
with Galois group $D_4$, we refer to $M$ as an {\em octic
  $D_4$-field}. In this article, we prove the following:

\begin{theorem}\label{main}
Let $N_{8}^{(\mathrm{r})}(D_4,X)$ $($respectively,
$N_{8}^{(\mathrm{c})}(D_4,X))$ denote the number of isomorphism classes
of totally real $($respectively, complex$)$ octic $D_4$-fields with absolute
discriminant bounded by $X$. Then
\begin{align*} N_{8}^{(\mathrm{r})}(D_4,X) &\sim \frac{56+3\sqrt{2}}{256} \cdot \prod_{p\geq 3}\Bigl(1+\frac{3}{p}+\frac{1}{p^{3/2}}\Bigr)\Bigl(1-\frac{1}{p}\Bigr)^3
\cdot X^{1/4} \log^2 X;\\ 
N_{8}^{(\mathrm{c})}(D_4,X) &\sim 
\frac{56+3\sqrt{2}}{128}
\cdot\prod_{p\geq 3}\Bigl(1+\frac{3}{p}+\frac{1}{p^{3/2}}\Bigr)\Bigl(1-\frac{1}{p}\Bigr)^3 \cdot
X^{1/4} \log^2 X.
\end{align*}
\end{theorem}

The above result concludes the strong form of Malle's conjecture for
Galois octic fields.
While we were working on Theorem \ref{main}, Koymans--Pagano \cite{KP} proved using different techniques that $$N_8(D_4,X) \ll c X^{1/4} \log^4(X)$$
for some (unknown) constant $c > 0$. Previous to either of these works, the best known result was verifying the weak form of Malle's Conjecture (see~\cite{MR1884706}), which was accomplished by Kluners--Malle \cite{MR2076117} in 2004. In particular, they proved that 
	$X^{1/4} \ll N_8(D_4,X) \ll X^{1/4 +\varepsilon}.$  
 
 In conjunction
with \cite{MR969545,habilitation}, both Theorem \ref{main} and \cite{KP} imply that 100\% of Galois octic
fields over $\bQ$ (when ordering such fields by their discriminants)
have Galois group isomorphic to $C_2 \times C_2 \times C_2$. Note that it does seem possible that by keeping track of local conditions during their argument, Koymans--Pagano's methods could achieve the sharp upper and lower bounds we prove here to obtain Malle's conjecture! However the authors of \cite{KP} remark that their techniques do not give an explicit handle for computing the constant of proportionality given in Theorem \ref{main}.

These constants in
Theorem \ref{main} satisfy the Malle-Bhargava principle for the
density of discriminants of number fields of a fixed Galois group, as predicted by recent work of Loughran--Santans \cite{LS}.
This is especially noteworthy given that the constant of proportionality for the
number of {\em quartic} $D_4$-fields, i.e., non-Galois quartic number fields
with degree 8 normal closure over $\bQ$, ordered by discriminant was determined by Cohen-Diaz y Diaz-Olivier \cite{MR1918290}) and it does
{\em not} satisfy the Malle-Bhargava principle (see \S5.3 of
\cite{D4preprint}).

\medskip

If a group $G$ has a transitive embedding into $S_n$, what we mean by an {\em $n$-ic
$G$-number field} is a degree $n$ extension $M$ of $\bQ$ such that the action of the Galois group of the
normal closure $\widetilde{M}$ of $M$ over $\bQ$ 
on the embeddings of $M$ into $\overline{\bQ}$ is isomorphic to
$G$. The {\em Malle-Bhargava principle} relies on the expectation that the
number of (isomorphism classes of) global $n$-ic $G$-number fields of
discriminant $D$ is proportional to the product indexed by $v$ of the ``weighted'' number of local
{\em $G$-compatible extensions} of $\bQ_v$ that are {\em discriminant-compatible}
with $D$, where $v$ ranges over all places $\nu(\bQ)$ of $\bQ$. Using this expectation, Bhargava \cite{MR2354798} conjectures the precise constant in the asymptotics of degree-$n$ $S_n$ fields. 

The Malle-Bhargava principle has
been verified in only a handful of cases. It is true for abelian $G$
with prime exponent due to work of Maki \cite{MR791087, MR1200974} and Wright
\cite{MR969545}.
It has been verified for degree $n$ $S_n$-fields when $n = 3$ by
Davenport-Heilbronn \cite{MR491593}, and $n = 4,5$ by
Bhargava \cite{MR2183288, MR2745272}, as well as for sextic
$S_3$-fields by Bhargava-Wood \cite{MR2373587}. However, the Malle-Bhargava principle does {\em not} hold in general, failing for example in the cases of cyclic fields of degree $p^2$ for any prime $p$; it also does {\em not} hold for
 {\em quartic} $D_4$-fields (proven by Cohen-Diaz y Diaz-Olivier
\cite{MR1918290}).

A more robust form of the Malle--Bhargava principle involves studying subfamilies of the original family constructed by imposing finitely many splitting conditions (see \cite{STCubicInv}).
To illustrate this, we obtain more general asymptotics for any family of octic $D_4$-fields satisfying any finite set of local conditions. For a place $v$ of $\bQ$, we define a {\em
  $D_4$-type} as a conjugacy class of continuous group homomorphisms
$\rho_v:\Gal(\overline{\bQ}_v/\bQ_v) \rightarrow D_4$.  An octic
\'etale algebra $M_v$ over $\bQ_v$ is {\em of $D_4$-type} if there
exists a continuous group homomorphism
$\rho_v:\Gal(\overline{\bQ}_v/\bQ_v) \rightarrow D_4$ such that $$M_v
\cong \Ind_{\im(\rho_v)}^{D_4} \overline{\bQ}_v^{\ker(\rho_v)},$$
i.e., $M_v$ is isomorphic to $\#\coker(\rho_v)$ copies of the fixed
field of $\ker(\rho_v)$ in $\overline{\bQ}_v$. In this case, we
further say that $M_v$ is {\em of $D_4$-type $[\rho_v]$}.

For a finite place $v$ of $\bQ$, a {\em local specification}
$\Sigma_v$ is a set consisting of pairs $(M_v,[\rho_v])$, where $M_v$
is (an isomorphism class of) an octic \'etale algebra of $\bQ_v$ of
$D_4$-type $[\rho_v]$. For a collection $\Sigma=(\Sigma_v)_v$ of local
specifications, we let $N_{8}(\Sigma,X)$ denote the number of
isomorphism classes of octic $D_4$-fields $M$ with absolute
discriminant bounded by $X$ such that $M \otimes \bQ_v \in \Sigma_v$
for all $v \in \Val(\bQ)$.
A collection $\Sigma = (\Sigma_v)_{v \in \Val(\bQ)}$ of local
specifications is said to be \emph{finite} if $\Sigma_p$ consists of
all $D_4$-types for large enough $p$. Then we have the following
result.

\begin{theorem}\label{congruence conditions}
Let $\Sigma$ be a finite collection of local specifications such that
$\Sigma_2$ contains all $D_4$-types. Then
\begin{equation*}
  N_{8}(\Sigma,X) \sim \frac{1}{4}\cdot \sum_{(M_\infty,\rho_\infty) \in
    \Sigma_\infty}\frac{1}{\#\Aut_{D_4}(\rho_\infty)}\cdot \prod_p \, \biggl(\bigl(1-{p}^{-1}\bigr)^3 \cdot
  \hspace{-7pt}\sum_{(M_p,\rho_p) \in
    \Sigma_p}\frac{|\Disc(M_p)^{\frac14}|_p}{\#\Aut_{D_4}(\rho_p)}\biggr)
  \cdot X^{\frac14}\log^2(X^{\frac14}).
\end{equation*}
where for all places $v$, $\Aut_{D_4}(\rho_v)$ is the
centralizer of $\im(\rho_v)$ in $D_4$.
\end{theorem}

It is a fundamental open problem in number field asymptotics to understand when
and why this local-global principle for the constant of
proportionality holds and when it fails. Wood \cite{MR2581243} has an illuminating discussion
of the abelian case.
Recent beautiful work of Loughran--Santens \cite{LS} discusses this question in its full generality, see in particular \cite[Remark 1.7]{LS}.
In this work, the authors give a theoretical explanation of when the Malle--Bhargava principle is expected to hold, giving in fact a conjectural leading term for the asymptotics of number fields with any fixed Galois group, ordered by any inertial invariant, and satisfying any fine set of local conditions (see \cite[Conjecture 9.10]{LS}). They predict that the family of octic $D_4$-fields {\em does} satisfy the Malle--Bhargava principle. Theorem \ref{congruence conditions} verifies their conjecture in this case.

\subsection{Background on relevant techniques}

Our proof techniques build on work of Cohen--Diaz y Diaz--Olivier \cite{MR1918290}, who
determine asymptotics for the number of quartic $D_4$-number fields,
ordered by discriminant. They observe that every such $D_4$-field $L$
is simply a quadratic extension of a quadratic field $K$. Moreover,
the discriminant of $L$ can be expressed as
\begin{equation*}
\Delta(L)=\Delta(K)^2\cdot N_{K/\Q}\Delta_K(L).
\end{equation*}
Thus, the number of quartic $D_4$-number fields with discriminant
bounded by $X$ can be expressed as a sum
\begin{equation*}
\sum_{[K:\Q]=2} \#\bigl\{L:[L:K]=2,\,N_{K/\Q}|\Delta_K(L)|<X/(\Delta(K)^2)\bigr\}.
\end{equation*}
For each fixed $K$, the summand in the above equation has
size $\asymp X/(\Delta(K)^{2+o(1)})$ by work of Datskowski--Wright
\cite{MR936994}. We should thus expect that only the quadratic fields $K$ with
``small discriminants'' (i.e., $|\Delta(K)|\ll X^\epsilon$) contribute to
the asymptotics. For these ``small'' quadratic fields, the asymptotic
of the corresponding summand is known (again by
Datskowski--Wright), and adding up the terms over quadratic fields with small discriminants does not introduce any large error terms. To obtain the result then, what is needed is a tail estimate on the number of $D_4$-fields with discriminants less than $X$, whose quadratic subfields $K$ have ``large discriminants'', i.e., $|\Delta(K)|\gg X^\epsilon$. Cohen--Diaz y Diaz--Olivier prove such an estimate and obtain their result.

In fact, the proof technique of \cite{MR1918290} generalizes quite widely.
Consider the question of counting composite extensions $L/K/\Q$, where
we have some control over the asymptotics of the number of fields $K$
(ordered by an invariant ${\rm I}(K)$) as well as control over the asymptotics of
relative extensions $L/K$, ordered by some invariant ${\rm J}_K(L)$, for
each fixed $K$. Suppose we are looking to count the number of such
extensions $L/\Q$, ordered by ${\rm Inv}(L)={\rm J}_K(L)\cdot {\rm I}(K)^\kappa$, where
$\kappa$ is sufficiently large. By sufficiently large, we mean that
$100\%$ of the extensions $L/K/\Q$ should be concentrated on the
``small'' subfields $K$, i.e., fields $K$ for which ${\rm I}(K)\ll
X^\epsilon$. In \cite{ALWWPreprint}, such families of number fields are termed {\it concentrated}. Determining the asymptotics of the number of fields $L$,
ordered by ${\rm Inv}(L)$, can be done in the following two steps: first,
for small subfields $K$, determine the exact count of
relative extensions $L/K$ and sum them up over the small fields, and second, prove a uniform bound on the number of fields $L$ with large subfields
$K$. Such uniform bounds are difficult to obtain in general. However, this
method has been used to great effect, see for example
\cite{MR2904935,MR4047213} and \cite{ALWWPreprint}, in which many cases of the strong form of Malle's conjecture are proved.

Note also that concentrated families of number fields of composite extensions are {\em not} be expected to satisfy the Malle--Bhargava principle. Indeed, we would expect these families to fail even the weaker condition of independence of splitting conditions at different primes. This was proven for the family of $S_3$-fields under many different orderings in \cite{STCubicInv}, and the proof generalizes to other concentrated families of composite extensions for which asymptotics are known without significant change.

In joint work with Altug and Wilson \cite{D4preprint}, we determine
the asymptotics of quartic $D_4$-fields $L$, ordered by the
Artin conductor ${\rm C}$ corresponding to the irreducible $2$-dimensional
representation of $D_4$. In this case, denoting the quadratic subfield
of $L$ by $K$, we have
\begin{equation*}
{\rm C}(L)=\Delta(K)\cdot \Nm_{K/\Q}\Delta_K(L).
\end{equation*}
In particular, the $\kappa$ from above is equal to $1$, and not
sufficiently large. As a consequence, it is no longer true that
$100\%$ of $D_4$-fields, when ordered by conductor, are
supported on quadratic fields with small discriminant. Every dyadic range for the
discriminant of $K$ contributes roughly equal amounts to the count of
$L$'s. The proof in \cite{D4preprint} then requires two new ingredients. First, after
counting quadratic extensions $L/K$ for fields $K$ with $\Delta(K)\leq
X^{1/2}$ (which can be done using the methods of \cite{MR1918290}), the argument relies on
reducing the count of $L/K$ with $\Delta(K)>X^{1/2}$ to the former count
via an algebraic input. Second, to explicitly determine the
asymptotic constants in order to verify the Malle--Bhargava principle in this conductor analogue, 
geometry--of--numbers methods (on the parameter space of quartic
rings) are applied to recover the same counts.

\subsection{Method of proof}

In this paper, we handle families of octic $D_4$-fields $M$ via
studying corresponding families of quadratic extensions $L$ of
quadratic fields $K$. To describe how the discriminant of $M$ relates
to the invariants of $L$ and $K$, we consider a slightly simplified
situation.
Let $\FF$ denote the set of Galois octic $D_4$-fields
$M$, such that $2$ does not ramify in $M$ and no rational prime $p$ totally
ramifies in $M$, i.e., $p\cO_M \neq \cP^8$ for any rational prime $p$ and any prime $\cP$ of $\cO_M$. For $M \in \cF$, denote a quartic
$D_4$-subfield of $M$ by $L$, and the quadratic subfield of $L$ by
$K$. Then, up to some bounded powers of $2$, we have
\begin{equation*}
\Delta(M)=\Delta(K)^4\cdot \rad(\Nm_{K/\Q}\Delta_K(L))^4,
\end{equation*}
where $\rad(n)$ denotes the radical of an integer $n$ or equivalently,
the product of primes dividing $n$. The presence of the radical
complicates our situation in multiple ways. First, since this is not a
polynomial invariant of the coefficients of the coregular space
$\Z^2\otimes\Sym^3(\Z^2)$ of quartic rings, geometry--of--number
methods, which rely on using an invariant whose definition can be extended to the real space $\R^2 \otimes \Sym^3(\R^2)$, are not available. Instead, for quadratic fields
$K$, we work with the Dirichlet series
\begin{equation*}
\Phi_K(s)=\sum_{[L:K]=2}\frac{1}{\Nm_{K/\Q}\Delta_K(L)^s}.
\end{equation*}
To count quadratic extensions $L$ of $K$ ordered by
$\rad(\Nm_{K/\Q}\Delta_K(L))$, we modify $\Phi_K(s)$ into a double
Dirichlet series: the Euler factor at $p$ of one of the ratios of
$L$-functions is of the form $1+c/p^s+d/p^{2s}$. We modify each such
Euler factor into $1+c/p^s+d/p^t$ yielding a double Dirichlet series
$\Phi_K(s,t)$.  Then the function $\Phi_K(s,2s)$ counts quadratic
extensions $L$ of $K$ ordered by $\Nm_{K/\Q}\Delta_K(L)$, while
$\Phi_K(s,s)$ counts quadratic extensions $L$ of $K$ ordered by
$\rad(\Nm_{K/\Q}\Delta_K(L))$.

The Dirichlet series $\Phi_K(s,2s)$ is a sum (of length $\asymp
\Cl(K)[2]$) of ratios of Hecke $L$-functions corresponding to
$K$. Similarly, $\Phi_K(s,s)$ can be expressed as a sum
of functions whose analytic properties are well understood. However,
two crucial difficulties arise when considering $\Phi_K(s,s)$ instead of $\Phi_K(s,2s)$:
\begin{itemize}
\item[{\rm (a)}] First: $V_4$-number fields
  $L/\Q$ contribute a positive proportion to this count, while 
  their numbers were easily seen to be negligible in the previous two orderings (quartic discriminant and Artin conductor).
\item[{\rm (b)}] Second: each nontrivial summand of $\Phi_K(s,s)$ has a simple pole at
  $s=1$, while the summand corresponding to the trivial character has
  a double pole at $s=1$. Previously, the summand of $\Phi_K(s,2s)$
  corresponding to the trivial character had a single pole at $s=1$,
  while all other summands were analytic to the right of $\Re(s)=1/2$.
\end{itemize}

Item (b) above introduces a very new phenomena to our problem: Note that
the difference between a double pole and a single pole corresponds to merely a
factor of $(\log X)$ in the asymptotics, while the difference between
a single pole and no pole corresponds to a postive power of $X$. As a consequence,
the ($\asymp\#\Cl(K)[2]$) summands to $\Phi_K(s,s)$ corresponding to
non-trivial Hecke characters contribute nontrivially to the main term
in the count of quartic fields $L/K$, ordered by $\Delta(K)\cdot
\rad(\Nm_{K/\Q}\Delta_K(L))$.

Surprisingly, these two difficulties can be simultaneously
resolved! We prove that the nontrivial characters contribute
predominantly to the count of $V_4$-fields $L$, with only a negligible
contribution to the count of $D_4$-fields. Furthermore, we prove that
the trivial character contributes predominantly to the count of
$D_4$-number fields. This situation is intriguingly parallel to that in Bhargava's landmark work \cite{MR2183288} determining asymptotics for the number of quartic $S_4$-number fields, in which he developed geometry-of-numbers methods and faced two major difficulties. First, $D_4$-fields constitute a positive proportion to the count of all quartic fields, and their contribution needs to be separated from that of $S_4$-fields. Second, the fundamental domain containing the lattice points parametrizing quartic rings had complicated cuspidal regions, and their contributions needed to be bounded or otherwise excluded from the main count. To resolve these (also simultaneously), Bhargava proved that $100\%$ of the $D_4$-fields are contained in the cuspidal regions and also that the points in the cuspidal regions do not contribute to the $S_4$-field count!

Therefore, to obtain asymptotics for
 $D_4$-fields, ordered by the discriminant of their Galois
closures, it suffices to study the asymptotics of the counting
functions arising only from the trivial Hecke characters of quadratic
fields. In this way, we are able to express the asymptotics of the
number of $D_4$-fields, ordered by discriminant, in terms of a sum
over $K$ of the double pole residue of $\Phi_K(s,s)$. Once this is done, we need a new technique to obtain asymptotics for this sum of residues. In our previous work counting $D_4$-fields by conductor, we had relied upon the parameter space of $D_4$-rings, and used geometry-of-numbers techniques to determine the asymptotic constant. In our case, this method is not available. Instead, we
relate the double pole residue of $\Phi_K(s,2s)$ to the single pole residue of
$\Phi_K(s,2s)$, use the results from~\cite{D4preprint} to evaluate the sum of these single pole residues over $K$, and thereby obtain our asymptotic constant.

\medskip

This paper is organized as follows: We begin by discussing families of
octic $D_4$-fields, of quartic $D_4$-fields, and of quadratic
extensions of quadratic fields in \S2. In \S3, we fix a quadratic
field $K$ and follow Wright \cite{MR969545} to express the counting
function of quadratic extensions of $K$ in terms of a sum involving
Hecke $L$-functions of $K$. We then construct the appropriate double
Dirichlet functions associated to these counting functions in \S4, and
study their analytic properties. Next, in \S5, we obtain  asymptotics for the number of quartic $D_4$-fields, ordered by a convenient height (which is related to the discriminant of the Galois closures of these quartic fields). Finally,
in \S6, we prove Theorems \ref{main} and \ref{congruence conditions}.

\subsection*{Acknowledgments}
\noindent Both authors were supported by NSERC discovery grants and Sloan fellowships. Shankar was also supported
by a Simons Fellowship. It is our pleasure to thank Brandon Alberts, S.~Ali Altu\u{g}, Manjul Bhargava, Dorian Goldfeld, Fatemehzahra Janbazi, Kiran Kedlaya, Robert Lemke Oliver, Daniel Loughran, Andres Macedo, Tim Santens, Jean-Pierre Serre, Alexander Slamen, Kevin H.~Wilson, and Melanie Matchett Wood for several useful discussions and helpful comments on a previous draft.

\section{Preliminaries}

Recall that $D_4$ denotes the order-8 group of symmetries of a
square. We let $\rotation$ denote a $90^\circ$-rotation
and $\reflection$ denote a reflection so that
\begin{equation*} D_4 =
\Bigl\langle \rotation ,\ \reflection \ \mid \ \rotation^4 =
1,\ \reflection^2 = 1, \ \reflection\rotation\reflection =
\rotation^3 \Bigr\rangle. \end{equation*} 
The group $D_4$ can also be interpreted as the wreath product of $C_2$
with itself, i.e.
	$$D_4 \cong C_2 \wr C_2 \cong V_4 \rtimes C_2,$$ where $C_2$
acts on $V_4 \cong C_2 \times C_2$ by permuting coordinates.

The group $D_4$ has a nontrivial outer automorphism $\outeraut$, which
can be described explicitly as follows:
\begin{equation*}
\outeraut(\rotation) = \rotation \quad \mbox{ and } \quad
\outeraut(\reflection) = \rotation\reflection.
\end{equation*}
Due to the existence of $\outeraut$, any octic $D_4$-field $M$
contains two conjugate pairs of non-isomorphic quartic $D_4$-fields.
If $M$ is an octic $D_4$-field, we will denote by $L=L_M$ the
(quartic) subfield fixed by the reflection subgroup $\langle
\reflection \rangle$. In addition, we denote the (quadratic) subfield
of $M$ fixed by $\langle \tau, \sigma^2 \rangle$ by $K=K_M$. With these
choices, $L$ is a quartic field whose normal closure over $\bQ$ is
equal to $M$, and $K$ is the quadratic subfield of $L$.

We say that $L$ is a {\em quartic $D_4$-field} if its normal closure
over $\bQ$ is an octic field $M=M_L$ with Galois group over $\bQ$ equal
to $D_4$. There is a quartic $D_4$-field contained in $M$ that is
not isomorphic to $L$. We denote (the isomorphism class of) this field
by $L'=\varphi_4(L)$, and its quadratic subfield by $K'=\varphi_2(L)$.

\subsection{Invariants of quadratic extensions of quadratic fields}

Let $K$ be a quadratic field and let $L$ be a quadratic extension of
$K$. In particular,
we employ the two fundamental invariants associated to $(L,K)$ in \cite{D4preprint}. First, let
\begin{equation*}
\q(L,K) := \Nm_{K/\bQ}(\Disc(L/K)) = \frac{|\Disc(L)|}{\Disc(K)^2}, \quad
\mbox{ and } \quad \D(L,K) := |\Disc(K)|,
\end{equation*}
where for any number field $F$, we let $\Disc(F)$ denote the discriminant of $F$. Additionally, for any subset $S$ of places in $\nu(\bQ)$ containing the infinite place, let $\Disc^S(F)$ denote the
prime-to-$S$-part of $\Disc(F)$. In general, for an invariant ${\rm I}
\in F$, and a set of places $S \subset \nu(F)$ containing the infinite
place of $F$, we let ${\rm I}^S$ denote the prime-to-$S$-part of
${\rm I}$.

Using splitting types, we can define more refined invariants which
allow us to decompose $\q(L,K)$ by conjugacy class of inertia when $L$ is a quartic $D_4$-field.
\begin{defn}\label{splittingtype}
If $F$ is a number field, then the {\em splitting type}
$\varsigma_p(F)$ at $p$ of $F$ satisfies
\begin{equation*}
  \varsigma_p(F) = (f_1^{e_1}f_2^{e_2} \hdots) \quad \Leftrightarrow \quad \cO_F/p\cO_F
  \cong \bF_{p^{f_1}}[t_1]/(t_1^{e_1}) \oplus \bF_{p^{f_2}}[t_2]/(t^{e_2}) \oplus \hdots
\end{equation*}
Similarly, if $R$ is a ring, then the {\em splitting type}
$\varsigma_p(R)$ at $p$ is equal to $(f_1^{e_1}f_2^{e_2} \hdots)$ if
and only if $$R/pR \cong \bF_{p^{f_1}}[t_1]/(t_1^{e_1}) \oplus
\bF_{p^{f_2}}[t_2]/(t^{e_2}) \oplus \hdots$$
\end{defn}

\begin{defn}\label{def:qs}
Let $K$ denote a quadratic field and $L$ denote a quadratic extension
of $K$. If $S \subset \Val(\bQ)$ denotes a finite set of rational
places containing $2$ and $\infty$, define
\begin{equation*}
\q^S_\sigma(L,K) := \prod_{\substack{ p \notin S \text{ s.t. } \\ \varsigma_p(L)=(1^4)}}
p, \qquad \qquad \q^S_{\sigma^2}(L,K) := \hspace{-10pt}
\prod_{\substack{ p \notin S \text{ s.t. } \\ \varsigma_p(L) \in \{(1^21^2),(2^2)\} \\ \varsigma_p(K) \in \{(11),(2)\}}} \hspace{-10pt}p,
\qquad \qquad \q^S_\tau(L,K) := \hspace{-10pt}\prod_{\substack{ p \notin S \text{ s.t } \\ \varsigma_p(L) \in \{(1^211),(1^22)\}}} \hspace{-10pt} p.
\end{equation*}
Finally, let $\D^S(L,K)$ denote the prime-to-$S$-part of $\D(L,K)$.
\end{defn}
When $L$ is a quartic $D_4$-field that is taken (without loss of generality) to be the fixed field of $\langle \tau \rangle$ inside its normal closure over $\bQ$ and $K$ denotes the quadratic subfield of $L$, we note that an odd prime $p \mid \q_{ g }^S(L,K)$ iff $I_p$ is conjugate to $\langle g \rangle$, where $g = \sigma, \sigma^2,$ or $\tau \in D_4$ and $p \notin S$.

\begin{defn}\label{sheight}
Define the $S$-{\it height} $\H=\H^S$ on such pairs $(L,K)$ of a quadratic
extension $L$ of a quadratic field $K$ to be
\begin{equation*}
\H^S(L,K):=\q_\sigma^{S}(L,K)^{1/2} \cdot \q^S_{\sigma^2}(L,K) \cdot \q^S_{\tau}(L,K) \cdot\D^{S}(L,K).
\end{equation*}
\end{defn}

When $L$ is a quartic $D_4$-field, we relate its $S$-height to the prime-to-$S$ part of the discriminants of $L$ and its normal closure over $\bQ$. 

\begin{lemma}\label{lem:Deltas}
Let $L$ is an quartic $D_4$-field with quadratic subfield $K$
and $M$ denotes the normal closure of $L$ over $\bQ$. If $S$ denotes a finite set of rational places containing $2$ and $\infty$, we have:
\begin{eqnarray*}
\Disc^S(M) \ &=& \H^S(L,K)^4 \\&=& \q^S_\sigma(L,K)^2\cdot \q^S_{\sigma^2}(L,K)^4 \cdot \q^S_\tau(L,K)^4 \cdot \D^S(L,K)^4 \\
\Disc^S(L) \  &=& \q^S_\sigma(L,K)^{\phantom 2} \cdot \q^S_{\sigma^2}(L,K)^2 \cdot \q^S_\tau(L,K)^{\phantom 2} \cdot \D^S(L,K)^2 
\end{eqnarray*}
\end{lemma}
\begin{proof}
  For an odd prime $p$, the $p$-power contribution to the
  discriminants of $M$, $L$, and the quadratic subfield of $L$, which
  we will denote by $K$, as well as to $\Nm_{K/\bQ}(\Disc(L/K))$ is
  determined entirely by the conjugacy class of the inertia group
  $I_p$ of $p$ in $D_4$. Because every odd prime is tame, $I_p$ must
  be a cyclic group; thus, the options (presented in their conjugacy classes) are $\{\langle 1 \rangle \}$,
  $\{\langle \tau \rangle,\langle\sigma^2 \tau\rangle\}$, $\{\langle
  \sigma \tau \rangle, \langle\sigma^3 \tau \rangle\}$,$\{\langle
  \sigma^2\rangle\}$, and $\{\langle \sigma \rangle\}$.

In the table below, we compute the $p$-divisibility of the
discriminants depending on the possible inertia groups within
$D_4$. Here, we assume (without loss of generality) that $L$ is fixed
by the reflection group $\langle \reflection \rangle$.
\begin{table}[h!!!!!!]\label{inertiatable}
\begin{center}
\begin{tabular}{|c||c|c|c|c|}
\hline &&&& \\[-10pt]
 $I_p$ &  $\Disc(M)$ &$\Disc(L)$ & $\Disc(K)$ &$\Nm_K(\Disc(L/K))$ \\ \hline &&&&\\[-12pt]
  		\hline &&&&\\[-10pt] 
$\langle1\rangle$ & $1$ & $1$ & $1$ & $1$ \\[-12pt] &&&&\\
\hline &&&& \\[-10pt]
$\langle \tau \rangle$, $\langle \sigma^2 \tau \rangle$ & $p^4$ & $p$ & $1$ & $p$ \\[2pt]
$\langle \sigma\tau \rangle$, $\langle \sigma^3 \tau \rangle$ & $p^4$ & $p^2$ & $p$ & $1$\\[-12pt] &&&&\\
\hline&&&&\\[-10pt]
$\langle \sigma \rangle$ & $p^6$ & $p^3$ & $p$ & $p$\\[2pt]
$\langle \sigma^2 \rangle$ & $p^4$ & $p^2$ & $1$ & $p^2$\\[-12pt] &&&&\\
\hline

\end{tabular}
\caption{$p$-power contribution to the invariants for an odd prime $p$.}
\end{center}\vspace{-15pt}
\end{table}

\noindent The proposition follows. 
\end{proof}

\subsection{Families of quadratic extensions of quadratic fields}

We will impose splitting conditions on quartic fields with quadratic subfields via
collections of local specifications, defined below.

\begin{defn}\label{localspec}
For each place $v$ of $\Q$, we use $\Lambda_v$ to denote a nonempty set of pairs
$(L_v,K_v)$, where $K_v$ is (an isomorphism class of) a quadratic \'etale extension of $\Q_v$
and $L_v$ is (an isomorphism class of) a quadratic \'etale extension of $K_v$. We refer to such $\Lambda_v$ as a set of {\em local specifications at $v$}. 

Let $p$ be a prime. 
\begin{itemize}
    \item A pair $(L_p,K_p)$ is said to be {\it totally ramified} if both the
extensions $K_p/\Q_p$ and $L_p/K_p$ are ramified.
    \item For a prime $p$, let $\Lambda_p^{\rm ntr}$ denote the set of all pairs $(L_p,K_p)$ that are not totally ramified.
    \item A collection of local specifications
$\Lambda=(\Lambda_p)_{p \in \nu(\bQ)}$ is said to be {\it acceptable} if for all
large enough $p$, the set $\Lambda_p$ contains every pair
$(L_p,K_p)$ which is not totally ramified, i.e., $\Lambda_p^{\rm ntr} \subseteq \Lambda_p$ for large enough $p$. 
    \item  A collection $\Lambda = (\Lambda_p)_{p \in \nu(\bQ)}$ is
said to have {\it finite total ramification} if for all large enough
$p$ the set $\Lambda_p$ has no pair $(L_p,K_p)$ which is totally
ramified.
\end{itemize}
Thus, if $\Lambda$ is an acceptable collection of local specifications with finite total ramification, then $\Lambda_p = \Lambda_p^{\rm ntr}$  for all $p$ large enough. Additionally: 
\begin{enumerate}
\item[\rm(a)]  Let $\cF_4(\Lambda)$ denote the set of pairs $(L, K)$
  consisting of (an isomorphism class of) a quadratic extension $L$ of
  a quadratic field $K$ over $\Q$ such that for every place $v \in \nu(\Q)$, we have
  $$(L_v,K_v):=(L \otimes_{\bQ} \bQ_v, K\otimes_{\bQ} \bQ_v) \in
  \Lambda_v.$$
\item[{\rm (b)}] Let $\FF_2(\Lambda)$ denote the set of quadratic
  extensions $K/\bQ$ such that there exists a quadratic extension $L$
  of $K$ such that $(L,K) \in \cF_4(\Lambda)$.
\item[{\rm (c)}] For a quadratic field $K \in \FF_2(\Lambda)$, let $\FF_4(\Lambda,K)$ be
  the set of $L$ such that $(L,K)\in\FF_4(\Lambda)$.
  \item[{\rm (d)}]   Let $S(\Lambda)\subset\Val(\bQ)$ 
consist of $2$, $\infty$, and all the odd primes $p$ such that
$\Lambda_p \neq \Lambda_p^{\rm ntr}$. (Note that since $\Lambda$ is acceptable with finite total
ramification, the set $S(\Lambda)$ is finite, and that the complement of $S(\Lambda)$ in the set of places of $\Q$ consisting of every odd prime $p$ such that $\Lambda_p = \Lambda_p^{\rm ntr}$. ) 

  \item[{\rm (e)}] For 
$K \in \FF_2(\Lambda)$, let $\cS_{K}(\Lambda)$ denote the valuations of $K$ above the
places in $S(\Lambda)$, and let
$\cS =\cS_K^{\rm ram}(\Lambda)$ denote the union of the set $\cS_K(\Lambda)$ and the set of
ramified places in $K$.
\item[{\rm (f)}] For $K\in\cF_2(\Lambda)$, let $\FF_4^\un(\Lambda,K)
  \subseteq \cF_4(\Lambda,K)$ be the subset of quadratic extensions
  of $K$ that are unramified away from $\cS_K(\Lambda)$.
\end{enumerate}
Finally, we call $(L,K) \in \FF_4(\Lambda)$ a {\em $D_4$-pair} if and only if $L$ is a quartic $D_4$-field. 
\end{defn}
Note that when $\Lambda$ is an acceptable collection with finite total ramification, $$\q_{\sigma}^{S(\Lambda)}(L,K) = 1$$ for every $(L,K) \in \FF_4(\Lambda)$ since $p \notin S(\Lambda)$ implies $p \nmid \q_{\sigma}(L,K)$.

\section{Quadratic extensions of a fixed quadratic field}

Fix an acceptable collection $\Lambda = (\Lambda_v)_{v\in \nu(\bQ)}$ of local
specifications with finite total ramification $\Lambda$ as in Definition \ref{localspec}, and let $K$ be a quadratic
field belonging to $\FF_2(\Lambda)$.
\begin{defn} We define the Dirichlet series
$\Phi_{\Lambda}(s;K)$ to be the counting function associated to the
set of $L$ in $\FF_4(\Lambda,K)$ ordered by
the prime-to-$S(\Lambda)$ norms of
their relative discriminants $\Disc(L/K)$:
\begin{equation}\label{Phi}
\begin{array}{rcl}
\Phi_{\Lambda}(s;K)&:=&\displaystyle\sum_{L
  \in\FF_4(\Lambda,K)}\Norm_{K/\bQ}^{S({\Lambda})}
(\Disc(L/K))^{-s}.
\end{array}
\end{equation}
\end{defn}
In this section, we obtain an explicit formula for
$\Phi_{\Lambda}(s;K)$ in terms of incomplete Hecke $L$-functions of
$K$ whose conductors are supported on primes in $S(\Lambda)$.
In order to state the result precisely, we will need to develop a bit more notation. 

  For each place $v$ of a quadratic field $K$, let
$K_v$ denote the completion of $K$ with respect to $v$, and when $v$
is finite, let $\cO_{K,v}^\times$ denote the ring of integers of
$K_v$. For an odd finite place $v$ of $K$, let $\chi_{K,v}$ denote the
nontrivial quadratic character on $\cO_{K,v}^\times$. A quadratic extension of $K$ unramified away
from a set of primes $\CS \subset \nu(K)$ is of the form $K(\sqrt{\alpha})$ for
$\alpha\in\O_{K,\CS}^\times$, the group of $\CS$-units of $K$. (Recall
that $\O_{K,\CS}^\times$ denotes the elements $x \in K$ such that the
fractional ideal $(x)$ is a product of nonzero powers of prime ideals
associated to the places in $\CS$.) For
$\alpha\in\O_{K,\CS}^\times$, we set
\begin{equation}\label{eqLKLambdaDef}
\begin{array}{rcl}
\zeta_{K,\CS}(s)&:=&\displaystyle
\prod_{\substack{v\in\val(K)\\v\not\in \CS}}
\Bigl(1-\frac{1}{\Nm_{K_v}(v)^s}\Bigr)^{-1};\\[.2in]
L_{K,\CS}(s;\alpha)&:=&\displaystyle
\prod_{\substack{v\in\val(K)\\v\not\in \CS}}
\Bigl(1-\frac{\chi_{K,v}(\alpha)}{\Nm_{K_v}(v)^s}\Bigr)^{-1}.
\end{array}
\end{equation}
We then prove the following result.
\begin{thm}\label{thmsec2main}
Let $\Lambda$ be an acceptable collection with finite total ramification as in Definition \ref{localspec} and recall the notation above in \eqref{eqLKLambdaDef}. We then have
\begin{eqnarray}\label{eqScountser}
  \Phi_{\Lambda}(s;K) &=& -1+B_\Lambda(K)
  \frac{\zeta_{K,\cS}(s)}{\zeta_{K,\cS}(2s)} +
  \sum_{K(\sqrt{\alpha})\in\FF^\un_{4}(\Lambda, K)}\hspace{-10pt} B_\Lambda(K,\alpha)
  \frac{L_{K,\cS}(s;\alpha)}{L_{K,\cS}(2s;\alpha)},
\end{eqnarray}
where $B_\Lambda(K)$ and $B_\Lambda(K,\alpha)$ are constants that
depend only on the splitting behavior at places in $\cS_K(\Lambda)$ of
$K$ and $K(\sqrt{\alpha})$, respectively, and are bounded by
$2^{|\cS_K(\Lambda)|}$.
\end{thm}

This section is organized as follows. First, in \S\ref{3.1} and \S\ref{3.2}, we
recall work of Wright \cite{MR969545} using class field theory to
parametrize quadratic extensions of $K$ and write $\Phi_\Lambda(s;K)$
as a sum of Euler products. In \S\ref{3.3}, we relate these Euler products to
$L_{K,\Lambda}(s;\alpha)$ and $\zeta_{K,\Lambda}(s)$.

\subsection{Parametrizing quadratic extensions of number fields
  via global class field theory}\label{3.1}

Let $K$ be a number field. Let $\A_K$ and $\A_K^\times$ denote the
ring of adeles of $K$ and the idele group of $K$, respectively. For a
finite set $\CS$ of places of $K$ containing all the infinite places,
let $\A_{K,\CS}$ and $\A_{K,\CS}^\times$ respectively denote the ring
of $\CS$-adeles and the group of $\CS$-ideles of $K$.  The ring of
$\CS$-integers, $\O_{K,\CS}$, consists of all $x \in K$ such that $x$
is integral away from prime ideals associated to places in $\CS$.
Let $\Cl_\CS(K)$ denote the $\CS$-class group of $K$, i.e., the
quotient of the class group of $K$ by the subgroup generated by the
ideal classes of the prime ideals of $\CS$. We begin with the
following lemma.

\begin{lemma}\label{lemcftosbij}
Let $K$ be a number field, and let $\CS$ be a finite set of places of
$K$ containing all infinite places such that $\Cl_{\CS}(K)$ is
odd. Then there is a natural isomorphism
\begin{equation*}
  \Hom(\A_K^\times/K^\times,\{\pm 1\})\cong
  \Hom(\A_{K,\CS}^\times/\O_{K,\CS}^\times,\{\pm 1\}).
\end{equation*}
\end{lemma}
\begin{proof}
We have the exact sequence
\begin{equation}\label{Sexactsequence}
0\longrightarrow\A_{K,\CS}^\times/\O_{K,\CS}^\times
\stackrel{\iota}{\longrightarrow}
\A_K^\times/K^\times\longrightarrow\Cl_\CS(K)\longrightarrow0.
\end{equation}
Since $\Cl_\CS(K)$ is odd by assumption, it follows that
$\Hom(\Cl_\CS(K),\{\pm 1\})=0$. Therefore, the map
\begin{equation}\label{eqinjectchar}
  \Hom(\A_K^\times/K^\times,\{\pm 1\})\to
  \Hom(\A_{K,\CS}^\times/\O_{K,\CS}^\times,\{\pm 1\}),
\end{equation}
obtained by composition, is an injection. 

To prove that it is a surjection, fix a character
$\chi:\A_{K,\CS}^\times/\O_{K,\CS}^\times\to\{\pm 1\}$. Let $n$ denote
the size of $\Cl_\CS(K)$. For any $x\in\A_K/K^\times$, the element
$x^{n}$ has trivial image in $\Cl_\CS(K)$, hence there exists $y_x \in
\A_{K,\CS}^\times/\O_{K,\CS}^\times$ such that $\iota(y_x)=
x^{n}$. Define the character
$\chi_0:\A_K^\times/K^\times\to\{\pm 1\}$ by
$$\chi_0(x): =\chi(y_x)$$ Now, if $z \in
\A_{K,\CS}^\times/\O_{K,\CS}^\times$ and $\iota(z) = x$, then $y_x =
z^{n}$. Since $n$ is odd, we have $$\chi_0 \circ \iota(z) = \chi(y_x)
= \chi(z^{|\Cl_\CS(K)|}) = \chi(z).$$ Hence, the injection described
in \eqref{eqinjectchar} is also a surjection, and so the lemma
follows.
\end{proof}

Every character $\chi:\A_{K,\CS}^\times\to\{\pm 1\}$ decomposes into a
product of local characters as
\begin{equation}\label{charlocprod}
 \chi=\prod_{v \in \nu(K)} \chi_v,\mbox{ where } \left\{
\begin{array}{lll}
  \chi_v:K_v^\times\to\{\pm 1\}      & \mbox{for } v\in \CS;\\[.05in]
  \chi_v:\O_{K,v}^\times\to\{\pm 1\}  & \mbox{for } v\not\in \CS.
\end{array}
\right.
\end{equation}
For $v\not\in \CS$, note that there are precisely two possible
characters for $\chi_v$ as any such character must factor through
$(\cO_{K,v}^\times)^2$ and
$\cO_{K,v}^\times/(\cO_{K,v}^\times)^2\cong \{\pm1\}$.  

Next, we have
the following parametrization of quadratic extensions of $K$.
\begin{proposition} \label{propbij}
Let $K$ be a number field and $\CS$ be a finite set of places of $K$
containing all infinite places such that $\Cl_\CS(K)$ is odd. Then
there is a bijection
\begin{equation*}
\Hom(\bA_{K,\CS}^\times/\cO_{K,\CS}^\times, \{\pm1\})
\longleftrightarrow \bigl\{\mbox{quadratic \'etale extensions of }
K\bigr\},
\end{equation*}
between the set of quadratic characters of
$\bA_{K,\CS}^\times/\cO_{K,\CS}^\times$ and the set of quadratic
\'etale extensions of $K$. Moreover, if $L$ is a quadratic extension
of $K$ corresponding to $\chi = \prod \chi_v$,
then:
\begin{itemize}
\item[{\rm (a)}] For $v\in\CS$, the field extension comprising the
  \'etale algebra $L_v:=L\otimes_K K_v$ corresponds to $\chi_v$ under
  the bijection of local class field theory;
\item[{\rm (b)}] For $v\not\in\CS$, the extension $L_v$ is unramified
  if and only if $\chi_v$ is trivial.
\end{itemize}
\end{proposition}
\begin{proof}
Global class field theory gives a bijection between (isomorphism
classes of) quadratic extensions of $K$ and index-2 subgroups of
$\bA_K^\times/K^\times$, the idele class group of $K$. Such index-2
subgroups are exactly the kernels of nontrivial characters $\chi_0:
\bA_K^\times/K^\times \rightarrow \{\pm 1\}$. The first assertion of
the proposition then follows from Lemma \ref{lemcftosbij}.

Now, let $\chi_0 = \prod_{v}\chi_{0,v}:\A_K^\times/K^\times\to\{\pm
1\}$ be the character associated to $L$ by class field theory. The
compatibility of local and global class field theory implies that (the
field extension of $K_v$ contained as a direct summand of)
$L_v:=L\otimes_K K_v$ corresponds to $\chi_{0,v}:K_v^\times\to\{\pm
1\}$ under the bijection of local class field theory.  Let
$\chi=\prod_v\chi_v:\A_{K,\CS}^\times/\O_{K,\CS}^\times\to\{\pm 1\}$
be the character corresponding to $\chi_0$ under the bijection of
Lemma~\ref{lemcftosbij}.  Part (a) of the second assertion of the
Proposition follows immediately since $\chi_v=\chi_{0,v}$ for $v\in
\CS$. For (b), note that $\chi_v=\chi_{0,v}|_{\O_{K,v}^\times}$ for
$v\not\in \CS$. In addition, $L_v/K_v$ is unramified if and only if
the tame inertia group is trivial. Under the local Artin map, the tame
inertia group is trivial if and only if the character $\chi_{0,v}$
restricted to $\O_{K,v}^\times$ is trivial, hence the proposition
follows.
\end{proof}

\subsection{Characters associated to  acceptable collections of local specifications}\label{3.2}

We return to the setting of Theorem \ref{thmsec2main}. Namely, let
$\Lambda$ be an acceptable collection of local
specifications with finite total ramification as in Definition \ref{localspec}. Let $K\in\cF_2(\Lambda)$ and recall that
$\cS  = \cS_K^{\rm ram}(\Lambda) = \cS_K(\Lambda)\cup\{v \in \nu(K) :v\mid\Delta(K)\}$ to be the set of places
in $K$ that are either above some place in $S(\Lambda)$ or
ramified. Recall  that $\cS$ includes all even and all infinite places
since $S(\Lambda)$ does.

Given a nontrivial character $\chi_v:K_v^\times\to\{\pm1\}$ of the completion of $K$ at $v$, recall
that $\chi_v$ corresponds to the degree-$2$ field extension of $K_v$
via local class field theory as in Proposition \ref{propbij}. If
$\chi_v$ is trivial, then we say that $\chi_v$ corresponds to the
\'etale algebra $K_v \oplus K_v$. 
\begin{defn}\label{ccdef}
Keep the notation as above. For $v\in\val(K)$:
\begin{itemize}
\item[{\rm (a)}] For $v\in\val(K)$ above a place in $S(\Lambda)$,
  define $\cC_{\Lambda,v}(K)$ to be the set of characters
  $\chi_v:K_v^\times\to\{\pm 1\}$ that correspond under Proposition \ref{propbij} to an algebra
  $L_v/K_v$ such that $(L_v,K_v)\in\Lambda_v$. Define
$\q_v(\chi_v):=1$ for each $\chi_v \in \cC_{\Lambda,v}(K)$.  
\item[{\rm (b)}] For a ramified place $v$ of $K$ that is not above any
  place in $S(\Lambda)$, define $\cC_{\Lambda,v}(K)$ to be the set of
  characters $\chi_v:K_v^\times\to\{\pm 1\}$ (including the trivial character) that correspond to all
  unramified quadratic extensions of $K_v$. Define
$\q_v(\chi_v):=1$ for each $\chi_v \in \cC_{\Lambda,v}(K)$.   
\item[{\rm (c)}] For a place $v$ of $K$ that does not belong to $\cS$,
  define $\cC_{\Lambda,v}(K)$ to be the set of all (two) characters
  $\chi_v:\cO_{K,v}^\times\to\{\pm1\}$. Define
\begin{equation*}
\q_v(\chi_v) := \begin{cases} 1 & \mbox{ if $\chi_v$ is
trivial, and} \\ p & \mbox{ if $\chi_v$ is the nontrivial
quadratic character.}
\end{cases}
\end{equation*}
(Note that since $v\nmid 2$, there are only two characters $\chi_v:
\cO_{K,v}^\times \rightarrow \{\pm 1\}$; the trivial character
$\chi_{1,v}$ and the unique nontrivial quadratic character
$\chi_{K,v}$.)
\end{itemize}
Let $\cC_{\Lambda}(K)$ denote the set of all quadratic characters
$\chi =\prod_v \chi_v: \bA^\times_{K,\cS} \rightarrow \{\pm 1\}$ such
that $\chi_v \in \cC_{\Lambda,v}(K)$ for all $v \in \cS$, and
let $\cC^{\ast}_{\Lambda}(K)$ denote the characters in
$\cC_{\Lambda}(K)$ that vanish on $\cO_{K,\cS}^\times$. 
For such a $\chi\in\cC_{\Lambda}(K)$, 
we define
\begin{equation}\label{qqdef}
\q(\chi):=\prod_{v\in\val(K)}\q_v(\chi_v).
\end{equation}
\end{defn}
\noindent The significance of the above definitions is from the
following result.
\begin{proposition}\label{corbij}
Let $\Lambda$ be an acceptable collection of local
specifications with finite total ramification as in Definition \ref{localspec}, and recall the notation in Definition \ref{ccdef}. If $K \in
\cF_2(\Lambda)$, then there is a natural bijection between the sets
$\cC^{\ast}_{\Lambda}(K)\setminus\{\rm triv\}$ and
$\FF_4(\Lambda,K)$ analogous to the bijection in Proposition \ref{propbij}. Moreover, if a character $\chi \in
\cC^{\ast}_{\Lambda}(K)$ corresponds to the quadratic extension $L$ of
$K$ under this bijection, then we have
$\q(\chi)=\q^{S(\Lambda)}(L,K)$.
\end{proposition}
\begin{proof}
The first assertion follows immediately from the definitions and
Propositions \ref{propbij}. To prove the second assertion, note that
$\q(\chi)$ and $\q^{S(\Lambda)}(L,K)$ are defined in terms of products
over $v\in\val(K)$ and $v\in\val(\Q)$, respectively. It is therefore
enough to prove that for all primes $p\in\val(\Q)$, we have
$\q_p(L,K)=\prod_v\q_v(\chi_v)$, where the product is over all
$v\in\val(K)$ above $p$. We do this case by case as follows.

First, if $p\in S(\Lambda)$, then $\q_p(L,K)=1$ and $\q_v(\chi_v)=1$
for every $v$ above $p$ (since in that case,
$v\in\cS_K(\Lambda)$. Next, assume that $p\notin S(\Lambda)$ (implying
$p\neq 2$) and that $p\nmid\Delta(K)$. In this case, we separately
consider the scenarios in which $p$ splits in $K$ and $p$ stays
inert. When $p$ splits as $v\bar{v}$, note that $L/K$ is ramified in
either none, one, or both of $v$ and $\bar{v}$, and that $\q_p(L,K)$
is either $1$, $p$, or $p^2$, respectively. Meanwhile, $\q_v(\chi_v)$
is either $1$ or $p$ depending on whether or not $\chi_v$ is trivial
or not on $\cO_{K,v}^\times$. The required equality then follows from
Part (b) of the second assertion in Proposition \ref{propbij}. The
case when $p$ stays inert in $K$ is identical.

Finally, we consider the case when $p\notin S(\Lambda)$ and
$p\mid\Delta(K)$. Denote the place in $K$ above $p$ by $v$. Since
$v\in\cS_K(\Lambda)$, it follows that $\q_v(\chi_v)=1$. Moreover, by
the definition of $\Lambda$, it follows that $v$ does {\it not} ramify
in $L$. Therefore, we have $\q_p(L,K)=1$. Altogether, we have proven
the proposition.
\end{proof}
By Proposition \ref{corbij} in conjunction with \eqref{Phi}, we have:
\begin{equation}\label{phichi}
\Phi_{\Lambda}(s;K) = -1+\sum_{\chi \in \cC^\ast_{\Lambda}(K)} \q(\chi)^{-s}.
\end{equation}
We would like to express $\Phi_{\Lambda}(s;K)$ as a sum over $\cC_\Lambda^\ast(K)$ so that it naturally factors into an Euler product. To do so, we follow Wright \cite{MR969545} and average over a set of representatives for $\O_{K,\cS}^\times/(\O_{K,\cS}^{\times})^2$. 
\begin{proposition}\label{propwrightexp}
Let $\Lambda$ be an acceptable collection of local
specifications with finite total ramification as in Definition \ref{localspec}, and recall the notation in Definition \ref{ccdef}. Hence, we let 
$\cS  = \cS_K(\Lambda)\cup\{v \in \nu(K) :v\mid\Delta(K)\}$, and additionally, we let
$\cA_{K,\cS}$ denote a set of representatives in $\O_{K,\cS}^\times$
for $\O_{K,\cS}^\times/(\O_{K,\cS}^{\times})^2$. If $K \in \cF_2(\Lambda)$, then we have
\begin{equation}\label{eqepgen}
  \Phi_{\Lambda}(s;K)=-1+\frac{1}{2^{|\cS|}}
  \sum_{\alpha\in\cA_{K,S}}\biggl(\;\prod_{{v\in \cS}}
  \Bigl(\sum_{ \chi_v \in \cC_{\Lambda,v}(K)}
  \chi_v(\alpha_v)\Bigr)\;\cdot
  \prod_{v\not\in \cS}\Bigl(1+\frac{\chi_{K,v}(\alpha_v)}{\Nm_{K_v}(v)^s}\Bigr)
  \biggl)
\end{equation}
where $\alpha_v$ denotes the image of $\alpha$ in $K_v^\times$ or $\cO_{K,v}^\times$, and $\chi_{K,v}$ denotes the unique nontrivial quadratic character
on $\O_{K,v}^\times$.
\end{proposition}
\begin{proof}
Recall that Dirichlet's unit theorem implies that $|\cA_{K,\cS}| =
2^{|\cS|}.$ For any character $\chi:\A_{K,\cS}^\times\to\{\pm 1\}$ in
$\cC_{\Lambda}(K)$, we have:
\begin{equation}\label{eqsumalpha}
\displaystyle\frac{1}{|\cA_{K,\cS}|}
\sum_{\alpha\in\cA_{K,\cS}}\chi(\alpha)\;\;=\;\;\left\{
\begin{array}{rl}
  1 &\mbox{if}\;\;\;\chi(\O_{K,\cS}^\times)=1;\\[.05in]
  0 &\mbox{if}\;\;\;\chi(\O_{K,\cS}^\times)\neq1;
\end{array}
\right.
\end{equation}
Rewriting the right hand side of the \eqref{phichi} using
\eqref{eqsumalpha}, we obtain
\begin{equation*}
  \Phi_{\Lambda}(s;K)=-1+
  \frac{1}{2^{|\cS|}}\sum_{\alpha\in\cA_{K,\cS}}
  \sum_{\chi \in \cC_{\Lambda}(K)}\frac{\chi(\alpha)}{\q(\chi)^s},
\end{equation*}
For each fixed $\alpha\in\cA_{K,\cS}$, let $\alpha_v$ denote the image of $\alpha$ in $K_v^\times$ for $v \in \cS$ or in $\cO_{K,v}^\times$ for $v \notin \cS$. Then the sum over $\chi \in
\cC_{\Lambda}(K)$ factors as:
\begin{eqnarray*}
  1+\Phi_{\Lambda}(s;K) &=& \frac{1}{2^{|\cS|}} \sum_{\alpha \in \cA_{K,\cS}}
  \biggl(\;\prod_{v \in \val(K)} \Bigl(\sum_{\chi_v \in \cC_{\Lambda,v}(K)}
  \frac{\chi_v(\alpha_v)}{\q_v(\chi_v)^s}\Bigr)\biggr)\\
&=&\frac{1}{2^{|\cS|}}
  \sum_{\alpha\in\cA_{K,\cS}}\biggl(\;\prod_{{v\in \cS}}
  \Bigl(\sum_{ \chi_v \in \cC_{\Lambda,v}(K)}
  \chi_v(\alpha_v)\Bigr)\;\cdot
  \prod_{v\not\in \cS}\Bigl(1+\frac{\chi_{K,v}(\alpha_v)}{\Nm_{K_v}(v)^s}\Bigr)
  \biggl),
\end{eqnarray*}
since $\q_v(\chi_v) = 1$ for $v \in \cS$, and there is a unique
nontrivial quadratic character for all $v \notin \cS$, namely, $\chi_{K,v}$ with $\q_v(\chi_{K,v}) = \Nm_{K_v}(v)$.
\end{proof}

\subsection{The counting function of quadratic extensions of $K$}\label{3.3}

We are now ready to prove Theorem \ref{thmsec2main}. Let
$\Lambda$ be an acceptable collection of local
specifications with finite total ramification as in Definition \ref{localspec}, recall Definition \ref{ccdef},
and additionally, we let
$\cA_{K,\cS}$ denote a set of representatives in $\O_{K,\cS}^\times$
for $\O_{K,\cS}^\times/(\O_{K,\cS}^{\times})^2$.

\medskip

\noindent {\it Proof of Theorem \ref{thmsec2main}}.
If $K \in \cF_2(\Lambda)$, a straightforward computation implies that for
$\alpha\in\cO_{K,\cS}^\times$, we have
\begin{equation*}
\begin{array}{rcll}
\displaystyle\prod_{v\not\in \cS}
\Bigl(1+\frac{1}{\Nm_{K_v}(v)^s}\Bigr) = \displaystyle\prod_{v\not\in \cS}
\Bigl(1+\frac{\chi_{K,v}(\alpha_v)}{\Nm_{K_v}(v)^s}\Bigr)&=&
\displaystyle\frac{\zeta_{K,\cS}(s)}{\zeta_{K,\cS}(2s)}&
\mbox{when }\alpha\in(\cO_{K,\cS}^\times)^2;\\[.2in]
\displaystyle\prod_{v\not\in \cS}
\Bigl(1+\frac{\chi_{K,v}(\alpha_v)}{\Nm_{K_v}(v)^s}\Bigr)&=&
\displaystyle\frac{L_{K,\cS}(s;\alpha)}{L_{K,\cS}(2s;\alpha)}&
\mbox{when }\alpha\in\cO_{K,\cS}^\times\backslash(\cO_{K,\cS}^\times)^2,
\end{array}
\end{equation*}
where $\alpha_v$ is the image of $\alpha$ in  $\cO_{K,v}^\times$ for $v \notin \cS$.
Proposition \ref{propwrightexp} and the above equations immediately
yield the following equality:
\begin{equation}\label{eqsec2mtalmost}
  \Phi_{\Lambda}(s;K)=
 -1+B_\Lambda(K)
  \frac{\zeta_{K,\cS}(s)}{\zeta_{K,\cS}(2s)} +
  \sum_{\alpha\in\cA_{K,\cS}}B_\Lambda(K,\alpha)
  \frac{L_{K,\cS}(s;\alpha)}{L_{K,\cS}(2s;\alpha)},
\end{equation}
where $B_\Lambda(K):=B_\Lambda(K,1)$ and $B_\Lambda(K,\alpha)$ are
defined to be
\begin{equation}\label{eqdefBLK}
  B_{\Lambda}(K,\alpha):=\pm\prod_{v\in\cS}
  \Bigl(\frac12\sum_{\chi_v\in\cC_{\Lambda,v}(K)}\chi_v(\alpha_v)
  \Bigr),
\end{equation}
where $\alpha_v$ is the image of $\alpha$ in $K_v^\times$ for $v \in \cS$, and the sign depends on whether or not $\alpha$ is a square at an
even or odd number of primes dividing $\Delta(K)$. It remains to
prove the following three points:
\begin{enumerate}
\item The sum over $\alpha$ can be restricted to those $\alpha$
  satisfying $K(\sqrt{\alpha}) \in \cF^{\un}_4(\Lambda,K)$.
\item The constants $B_{\Lambda}(K)$ and $B_{\Lambda}(K,\alpha)$
  depend only on the splitting behavior of $K$ and $K(\sqrt{\alpha})$,
  respectively, at primes in $\cS_K(\Lambda)$.
\item We have the bounds $|B_{\Lambda}(K)|,|B_{\Lambda}(K,\alpha)|\ll
  2^{|\cS_K(\Lambda)|}$.
\end{enumerate}
To that end, let $v$ be an odd nonarchimedean place of $K$, and let
$p$ denote the rational prime below $v$. A character
$\chi_v:K_v^\times\to\{\pm 1\}$ factors through
$K_v^\times/(K_v^{\times})^2$, and since $v$ is odd, we have that
$K_v^\times/(K_v^{\times})^2 \cong V_4$, the Klein four
group. Furthermore, a set of representatives for
$K_v^\times/(K_v^{\times})^2$ is $\{1,n_v,\pi_v,n_v\pi_v\}$, where
$n_v\in\O_{K,v}^\times$ is a non-residue, and $\pi_v\in K_v^\times$ is
a prime element. There are four characters
$\chi_v:K_v^\times\to\{\pm1\}$, which we have enumerated below:
\begin{table}[h!!!!!!!]\begin{center} \begin{tabular}{|c||c||c|c|c|c|}
\hline
Local Character $\chi_v$  & $\q_v(\chi_v)$ & $\chi_v(1)$ & $\chi_v(n_v)$ & $\chi_v(\pi_v)$ & $\chi_v(n_v\pi_v)$
\\\hline \hline
$\chi_{1,v}$ & $1$ & $1$ & $1$ & $1$ & $1$
\\\hline
$\chi_{n_v}$ & $1$ & $1$ & $1$ & $-1$ & $-1$
\\\hline
$\chi_{\pi_v}$ & $p$ & $1$ & $-1$ & $1$ & $-1$
\\\hline $\chi_{n_v\pi_v}$ & $p$ & $1$ & $-1$ & $-1$ & $1$
\\\hline \end{tabular}
\caption{\small Local quadratic characters of $K_v^\times$ when $v$ is
  odd and nonarchimedean \label{ramchartable}}
\end{center}\vspace{-15pt}
\end{table}

Now, recall that if $v$ is a ramified place of $K$ such that
$v\not\in\cS_K(\Lambda)$, then $\cC_{\Lambda,v}(K)$ consists of the
two unramified characters $\chi_{1,v}$ and $\chi_{n_v}$. Thus, for $\alpha\in\cA_{K,\cS}$, if $\alpha_v$ denotes the image of $\alpha$ in $K_v^\times$,
Table \ref{ramchartable} implies that \begin{equation}\label{removingoddalpha}
\sum_{\chi_v\in\cC_{\Lambda,v}(K)}{\chi_v(\alpha_v)} = 
\left\{
\begin{array}{ccl}
\displaystyle 2
&\mbox{if}& {\rm val}_v(\alpha) \mbox{ is even};\\[.1in]
0&\mbox{if}& {\rm val}_v(\alpha) \mbox{ is odd}.
\end{array}
\right.
\end{equation}
Hence, if ${\rm val}_v(\alpha)$ is odd for any
ramified $v\not\in\cS_K(\Lambda)$, then
$B_\Lambda(K,\alpha)=0$, which implies that the sum over $\alpha$ in $\Phi_{\Lambda}(s;K)$ can be restricted to those $\alpha$ satisfying $K(\sqrt{\alpha}) \in \cF^{\un}_4(\Lambda,K)$. Moreover, if ${\rm val}_v(\alpha)$ is even for
every ramified $v\not\in\cS_K(\Lambda)$, then
we can conclude that
\begin{equation}\label{eqeqBLK}
  B_{\Lambda}(K,\alpha)= \pm \prod_{v\in\cS_K(\Lambda)}
  \Bigl(\frac12\sum_{\chi_v\in\cC_{\Lambda,v}(K)}\chi_v(\alpha)
  \Bigr).
\end{equation}
This implies that the constants $B_{\Lambda}(K)$ and $B_{\Lambda}(K,\alpha)$ depend only on the splitting behavior of $K$ and $K(\sqrt{\alpha})$, respectively, at primes in $\cS_{K}(\Lambda)$. Finally, Equation \eqref{eqeqBLK} in conjunction with Table \ref{ramchartable} implies that $|B_{\Lambda}(K)|$ and $|B_{\Lambda}(K,\alpha)|$ are bounded by $2^{|\cS_K(\Lambda)|}$. This concludes the proof of Theorem
\ref{thmsec2main}. \qed

\medskip

\begin{remark}\label{remTR}  Let
$\Lambda$ be an acceptable collection of local
specifications with finite total ramification as in Definition \ref{localspec}, and recall Definition \ref{ccdef}.
Suppose that $p \notin S(\Lambda)$ is an odd prime, and $\Lambda_p$ is replaced with  the set $\Lambda_p^{\rm tr}$ of all pairs of totally ramified local algebras. Denote this new acceptable collection of local specifications $\Lambda'$. Hence, for $K \in \cF_2(\Lambda) \cap \cF_2(\Lambda')$, we have $p \mid \Delta(K)$. For  a
place $v$ of $K$ above $p$, we have
$\cC_{\Lambda,v}(K)=\{\chi_{\pi_v},\chi_{n_v\pi_v}\}$. Therefore, we
may compute the factor of $B_\Lambda(K,\alpha)$ corresponding to $v$
using Table \ref{ramchartable}:
\begin{equation*}
\frac{1}{2}\sum_{\chi_v\in \cC_{\Lambda,v}(K)}\chi_v(\alpha)=\left\{
\begin{array}{rcl}
  0 &\mbox{if}& v\mid\alpha;\\
  1 &\mbox{if}& v\nmid\alpha\mbox{ and $\alpha$ is a residue}\\
  -1 &\mbox{if}& v\nmid\alpha\mbox{ and $\alpha$ is a nonresidue}.\\
\end{array}
\right.
\end{equation*}
In particular, forcing ramification above ramified places of $K$ does
not increase the size of $B_\Lambda(K,\alpha)$, and does not increase
the number of $\alpha$ that we must sum over in order to analyze
$\Phi_\Lambda(s;K)$.
\end{remark}

\section{Analytic properties of certain useful double Dirichlet series}

Let $\Lambda$ be an acceptable collection of local 
specifications with finite total ramification as in Definition \ref{localspec}, and let $K \in
\cF_2(\Lambda)$ be a quadratic field. Let $K(\sqrt{\alpha}) \in
\cF_4^{\un}(\Lambda,K)$ be a quadratic extension of $K$ that is
unramified outside of the places above $\cS_K(\Lambda)$. As before,
let $\cS = \cS_K(\Lambda) \cup \{v \in \val(K) : v\mid\Delta(K)\}$,
and for a place $v\not\in\cS$ of $K$, recall that $\chi_{K,v}$ denotes
the (unique) nontrivial quadratic character on $\cO_{K,v}^\times$.
Recall also that $\zeta_{K,\cS}(s) = L_{K,\cS}(s;1)$ as in \eqref{eqLKLambdaDef}.

The local Euler factors of $L_{K,\Lambda}(s; \alpha)/L_{K,\Lambda}(2s;
\alpha)$, when indexed by rational primes, is given as follows: If $p\not\in S:=S(\Lambda) \cup \{p \in
\val(\bQ) : p \mid \Delta(K)\}$, then
the local Euler factor at $p$ is
\begin{equation}\label{eqpcont}
\frac{L_{K,\cS}(s;\alpha)_p}{L_{K,\cS}(2s;\alpha)_p}=\left\{
\begin{array}{rcl}
  \displaystyle\Bigl(1+\frac{\chi_{K,v}(\alpha)+\chi_{K,\overline{v}}(\alpha)}{p^s}+
  \frac{\chi_{K,v}(\alpha)\chi_{K,\overline{v}}(\alpha)}{p^{2s}} \Bigr)
&{\rm if}& p = v\overline{v} \mbox{ splits in $K$}; \\[.15in]
  \displaystyle\Bigl(1+\frac{\chi_p(\alpha)}{p^{2s}} \Bigr)
  &{\rm if}& p \mbox{ is inert}\mbox{ in $K$}.
\end{array}\right.
\end{equation}
If $p \in S$, then the local Euler factor at $p$ is simply $1$. This
leads us to the following definition.
\begin{defn}\label{dirichletfunction}
Keep the notation as above. For a quadratic field $K \in \cF_2(\Lambda)$ and a quadratic extension
$K(\sqrt{\alpha}) \in \cF_4^{\un}(\Lambda,K)$ that is unramified outside of
the places above $\cS_K(\Lambda)$, we define the double Euler factor
at primes $p\not\in S$ to be:
\begin{equation}\label{twovarlocalfactor}
\cD_{K,\Lambda}(s,t;\alpha)_p:= 
\left\{
\begin{array}{rcl}
\displaystyle\Bigl(1+\frac{\chi_{K,v}(\alpha)+\chi_{K,\overline{v}}(\alpha)}{p^t}+
\frac{\chi_{K,v}(\alpha)\chi_{K,\overline{v}}(\alpha)}{p^{s}} \Bigr)
&{\rm if}& p = v\overline{v} \mbox{ splits in $K$;} \\[.15in]
\displaystyle\Bigl(1+\frac{\chi_p(\alpha)}{p^{s}} \Bigr)
&{\rm if}& p \mbox{ is inert in $K$},
\end{array}\right.
\end{equation}
and define the Euler product
\begin{equation}\label{twovarzfunctiontriv}
  \cD_{K,\Lambda}(s,t;\alpha):=\prod_{p\notin S}
  \cD_{K,\Lambda}(s,t;\alpha)_p.
\end{equation}
Define
$\cD_{K,\Lambda}(s,t) := \cD_{K,\Lambda}(s,t;1)$
and let
\begin{equation}\label{twovarphi}
\Phi_{\Lambda}(s,t;K) := \displaystyle\sum_{L\in
  \FF_4(\Lambda,K)}
\q_{\sigma^2}^{S(\Lambda)}(L,K)^{-s}\cdot\q_{\tau}^{S(\Lambda)}(L,K)^{-t}.
\end{equation}
\end{defn}

Analogous to Theorem \ref{thmsec2main}, we have the following result.
\begin{cor}\label{eqScounttrip}
If $\Lambda$ is an acceptable collection with finite total ramification as in Definition \ref{localspec} and $K \in \cF_2(\Lambda)$, then we have
\begin{equation}\label{eqcorparameq}
\Phi_{\Lambda}(s,t;K) = -1+
\displaystyle B_\Lambda(K)\cdot \cD_{K,\Lambda}(s,t) + 
\sum_{K(\sqrt{\alpha})\in\FF^\un_{4}(\Lambda,K)}B_\Lambda(K,\alpha)
\cdot\cD_{K,\Lambda}(s,t;\alpha).
\end{equation}
\end{cor}
\begin{proof}
Write $\Phi_\Lambda(s,t;K)$ in its Dirichlet series expansion as
\begin{equation*}
\Phi_\Lambda(s,t;K):= \frac{a_{m,n}}{m^sn^t},
\end{equation*}
and note the following properties of the Dirichlet coefficient $a_{m,n}$. First, we have $a_{m,n}=0$ unless $(m,n)=1$. This follows from Definition \ref{def:qs} which immediately implies that $(q_{\sigma^2}^{S(\Lambda)}(L,K),q_\tau^{S(\Lambda)}(L,K))=1$. Second, if $(m,n)=1$, then we have $a_{m,n} = b_{m^2n}$, where $b_{m^2n}$ denotes the $m^2n$th Dirichlet coefficient of $\Phi_\Lambda(s;K)$ defined in \eqref{Phi}. This follows from Lemma \ref{lem:Deltas}, which implies that we have 
\begin{equation*}
\Norm_{K/\Q}^{S(\Lambda)}(\Delta(L,K)) = \big(\q_{\sigma^2}^{S(\Lambda)}(L,K)\big)^2\q_\tau^{S(\Lambda)}(L,K).
\end{equation*}
The desired result now follows from Theorem \ref{thmsec2main}.
\end{proof}

In this section, we study the analytic properties of these double
Dirichlet series. In \S4.1, we study the analytic continuation, poles,
and residues of $\cD_{K,\Lambda}(s,t;\alpha)$. In \S4.2, we introduce
standard methods of counting smoothed sums of Dirichlet coefficients
of $L$-functions, and of estimating special values. Then in \S4.3, we
use the previous methods to estimate smoothed sums of the coefficients
of this double Dirichlet series in various ranges.

\subsection{Analytic continuation and poles of $\cD_\Lambda(s,t;K,\alpha)$}

We introduce the following $L$-functions to analyze the double
Dirichlet series $\cD_\Lambda(s,t;K,\alpha)$.
\begin{defn}
Let $\Lambda$ be an acceptable collection of local 
specifications with finite total ramification as in Definition \ref{localspec}, and let $K \in
\cF_2(\Lambda)$ be a quadratic field. Let $K(\sqrt{\alpha}) \in
\cF_4^{\un}(\Lambda,K)$ be a quadratic extension of $K$ that is
unramified outside of the places above $\cS_K(\Lambda)$. If $\chi_{K,v}$ denotes the (unique) nontrivial quadratic character on $\cO_{K,v}^\times$, set
$\chi_\alpha(v)=\chi_{K,v}(\alpha)$, and when a rational prime $p$
splits as $v\overline{v}$, set
$\chi_\alpha(p)=\chi_\alpha(v)\chi_\alpha(\overline{v})$.  We set
$\cS:=\cS_K(\Lambda)\cup\{v:v\mid\Delta(K)\}$ to be the set of places
in $K$ that are either above some place in $S(\Lambda)$ or
ramified, and let $S:=S(\Lambda) \cup \{p \in
\val(\bQ) : p \mid \Delta(K)\}$. We define
the two (truncated) $L$-functions
\begin{equation}\label{rationalquadraticLfunction}
\begin{array}{rcl}
L_{\Lambda,\Q}(s;K,\alpha)&:=&\displaystyle
\prod_{\substack{p\in\val(\Q)\\ p\not\in S}}
\Bigl(1-\frac{\chi_\alpha(p)}{p^s}\Bigr)^{-1};
\\[.2in]    
L_{\Lambda,K}(t;\alpha)&:=&\displaystyle
\prod_{\substack{\fp\in\val(K)\\\fp\not\in\cS}}
\Bigl(1-\frac{\chi_\alpha(\fp)}{\Nm(\fp)^t}\Bigr)^{-1}.
\end{array}
\end{equation}
For a quadratic field $F$, and for an element $\alpha\in F \setminus \bQ$, 
we define the (completed) Hecke $L$-functions
\begin{equation*}
L(s,\chi_F):=\prod_{p\nmid\Delta(F)}\Bigl(1-\frac{\chi_F(p)}{p^s}\Bigr)^{-1},\quad
L(s,\chi_\alpha):=\prod_{v\nmid\Delta(F)}\Bigl(1-\frac{\chi_\alpha(v)}{\Nm(v)^s}\Bigr)^{-1},
\end{equation*}
where $\chi_F(p)=1$ when $p$ splits in $F$ and $\chi_F(p)=-1$ when $p$ is
inert in $F$.
\end{defn}
When $\alpha \in \cA_{K,\Lambda}$ is trivial, we denote
$L_{\Lambda,\Q}(s;K,1)$ by $\zeta_{\Lambda,\Q}(s;K)$ which is simply a
truncated Riemann zeta function, and we denote $L_{\Lambda,K}(t;1)$ by
$\zeta_{\Lambda,K}(t)$ which is a truncated Dedekind zeta
funtion. When $\alpha$ is nontrivial, the functions
$L_{\Lambda,\Q}(s;K,\alpha)$ and $L_{\Lambda,K}(t;\alpha)$ are
truncated versions of well known Hecke $L$-functions.

We have the following lemma regarding $\chi_\alpha$, whose proof is
immediate from the definitions.
\begin{lemma}\label{lemspV4}
Let $\Lambda$ be an acceptable collection of local 
specifications with finite total ramification as in Definition \ref{localspec}, and let $K \in
\cF_2(\Lambda)$ be a quadratic field. Let $K(\sqrt{\alpha}) \in
\cF_4^{\un}(\Lambda,K)$ be a quadratic extension of $K$ that is
unramified outside of the places above $\cS_K(\Lambda)$.  Finally, let $\cS = \cS_K(\Lambda) \cup \{v \in \val(K) : v\mid\Delta(K)\}$. Then for a place $v\not\in \cS$ of $K$, we
have $\chi_\alpha(v)=1$ when $v$ splits in $K(\sqrt{\alpha})$ and
$\chi_\alpha(v)=-1$ when $v$ is inert in $K(\sqrt{\alpha})$.
\end{lemma}

Next we describe the functions $L_{\Lambda,\Q}(s;K,\alpha)$ and
$L_{\Lambda,K}(t;\alpha)$ in terms of  ``completed'' Hecke $L$-functions,
and list their conductors.

\begin{proposition}\label{lemLQpole}
Let $\Lambda$ be an acceptable collection of local 
specifications with finite total ramification as in Definition \ref{localspec}, and let $K \in
\cF_2(\Lambda)$ be a quadratic field. Let $K(\sqrt{\alpha}) \in
\cF_4^{\un}(\Lambda,K)$ be a quadratic extension of $K$ that is
unramified outside of the places above $\cS_K(\Lambda)$.  We let $P(\Lambda)$ denote square
of the product of the finite primes in $S(\Lambda)$.  Away from the primes in $S=S(\Lambda) \cup \{p \in
\val(\bQ) : p \mid \Delta(K)\}$, the functions
$L_{\Lambda,\Q}(s;K,\alpha)$ and $L_{\Lambda,K}(t;\alpha)$ are as
given in Table \ref{table}. Moreover, the conductors of $\zeta_K(t)$
and $L(s,\chi_K)$ are $|\Delta(K)|$, and the conductors of
$L(t,\chi_\alpha)$ and $L(s,\chi_{\varphi_2(K(\sqrt{\alpha}))})$ are
bounded by $P(\Lambda)|\Delta(K)|$.
\begin{table}[ht]
\centering
\begin{tabular}{|c | c| c| c|c|}
\hline $\Gal(\widetilde{K(\sqrt{\alpha})}/\bQ)$ & $C_2$ & $V_4$ &
$C_4$ & $D_4$\\
\hline \hline
$L_{\Lambda,\Q}(s;K,\alpha)$ & $\zeta(s)$ & $\zeta(s)$ &
$L(s,\chi_K)$ &
$L(s,\chi_{\varphi_2(K(\sqrt{\alpha}))})$\\
\hline
$L_{\Lambda,K}(t;\alpha)$ & $\zeta_K(t)$ &
$L(t;\chi_\alpha)$ &
$L(t;\chi_\alpha)$ &
$L(t;\chi_\alpha)$ \\
\hline
\end{tabular}
\caption{Completed $L$-functions}\label{table}
\end{table}
\end{proposition}
(Recall that for a quartic $D_4$
field $L$, we denote by $\varphi_2(L)$ the quadratic subfield
contained inside a non-isomorphic quartic $D_4$-field $\varphi_4(L)$
with the same normal closure as $L$ over $\bQ$.)
\begin{proof}
The claims regarding the description of the functions
$L_{\Lambda,K}(t;\alpha)$ follow immediately from the definitions. The
claims regarding the conductors in these cases can be deduced by noting
that $K$ and $\chi_\alpha$ are unramified outside places in $\cS= \cS_K(\Lambda) \cup \{v \in \val(K) : v\mid\Delta(K)\}$.

\medskip

We now turn to the functions $L_{\Lambda,\Q}(s;K,\alpha)$. When
$K(\sqrt{\alpha})=K$, the result is immediate. Next, assume that
$K(\sqrt{\alpha})$ is a biquadratic field. For a prime $p \notin S$,
$\varsigma_p(K(\sqrt{\alpha})) = (1111)$ or $(22)$. That
$L_{\Lambda,\Q}(s;K,\alpha)$ is as given in Table \ref{table} follows
from Lemma \ref{lemspV4}, which implies that $\chi_\alpha(p)=1$ for
all such primes $p$.

If $K(\sqrt{\alpha})$ is a cyclic quartic field, then the possible
splitting types of $K(\sqrt{\alpha})$ and $K$ at primes $p \notin S$
are:
\begin{itemize}
\item[{\rm (a)}] $\varsigma_p(K_\alpha) = (1111)$;
  and $\varsigma_p(K) = (11);$%\vspace{-5pt}
\item[{\rm (b)}] $\varsigma_p(K_\alpha) = (22)$;
  and $\varsigma_p(K) = (11);$%\vspace{-5pt}
\item[{\rm (c)}] $\varsigma_p(K_\alpha) = (4)$ and $\varsigma_p(K) = (2)$. 
\end{itemize}
Hence, by Lemma \ref{lemspV4}, we we see that $\chi_\alpha(p)=1$ iff
$\varsigma_p(K) = (11)$ and $\chi_\alpha(p)=-1$ iff
$\varsigma_p(K)=(2)$. The required result thus follows.

Finally, if $K_\alpha$ is a quartic $D_4$-field, then the possible
splitting types of primes $p \notin S$ include the three listed above
as well as:
\begin{itemize}
\item[{\rm (d)}] $\varsigma_p(K_\alpha) = (112)$
  and $\varsigma_p(K) = (11);$%\vspace{-5pt}
\item[{\rm (e)}] $\varsigma_p(K_\alpha) = (22)$ and $\varsigma_p(K) =
  (2)$.
\end{itemize}
Hence, we have $\chi_\alpha(p)=1$ iff $\varsigma_p(K_\alpha) = (1111)$ or
$(22)$ by Lemma \ref{lemspV4}. Note that this implies that we have
$\chi_\alpha(p) = 1$ iff $\varsigma_p(\varphi_2(K(\sqrt{\alpha}))) =
(11)$, and $\chi_\alpha(p)=-1$ otherwise. The claim regarding 
the description of $L_{\Lambda,\Q}(s;K,\alpha)$ thus follows.

Finally, we note that the required results regarding the conductors of
$L(s,\chi_K)$ and $L(s,\chi_{\phi_2(K(\sqrt{\alpha})})$
follow by noting that $K$ and $\varphi_2(K(\sqrt{\alpha}))$ are ramified only
at places dividing $K$ and at places in $S$, respectively.
\end{proof}

The preceding result allows us to analytically continue the functions
$\cD_\Lambda(s,t;K,\alpha)$.
\begin{thm}\label{proppole2}
Let $\Lambda$ be an acceptable collection of local 
specifications with finite total ramification as in Definition \ref{localspec}, and let $K \in
\cF_2(\Lambda)$ be a quadratic field. Let $K(\sqrt{\alpha}) \in
\cF_4^{\un}(\Lambda,K)$ be a quadratic extension of $K$ that is
unramified outside of the places above $\cS_K(\Lambda)$.
The functions $\cD_\Lambda(s,t;K,\alpha)$ are analytic in the region
$\{\Re(s)>\frac12,\Re(t)>\frac12\}$ with the exception of possible
simple polar divisors at $s=1$ and $t=1$. Specifically, we have
\begin{enumerate}
\item If $\chi_\alpha$ is trivial $($i.e., $K(\sqrt{\alpha})=K)$, then
  $\cD_\Lambda(s,t;K,\alpha)$ has polar divisors at $s=1$ and $t=1$. 
\item If $K(\sqrt{\alpha})$ is a biquadratic field, then
  $\cD_\Lambda(s,t;K,\alpha)$ has exactly one polar divisor at $s=1$.
\item If $K(\sqrt{\alpha})$ is a cyclic quartic field
  or a quartic $D_4$-field, then $\cD_\Lambda(s,t;K,\alpha)$ is holomorphic in
  the region $\{\Re(s)>\frac12,\Re(t)>\frac12\}$.
\end{enumerate}
\end{thm}
\begin{proof}
For $p\not\in S=S(\Lambda) \cup \{p \in
\val(\bQ) : p \mid \Delta(K)\}$, the Euler factor at $p$ is denoted by
$F_{\Lambda,p}(s,t;K,\alpha)$ for the function
\begin{equation}\label{eqFfunc}
F_\Lambda(s,t;K,\alpha):=\frac{\cD_\Lambda(s,t;K,\alpha)}
{L_{\Lambda,\Q}(s;K,\alpha)L_{\Lambda,K}(t;\alpha)},
\end{equation}
and it is easily computed to be:
\begin{equation}\label{eqFfunc1}
  F_{\Lambda,p}(s,t;K,\alpha)=\Bigl(
  1+\frac{O(1)}{p^{2s}}+\frac{O(1)}{p^{s+t}}+\frac{O(1)}{p^{2t}}
  +\frac{O(1)}{p^{3t}}+\frac{O(1)}{p^{2s+t}}+\frac{O(1)}{p^{s+2t}}
  +\frac{O(1)}{p^{s+3t}}+\frac{O(1)}{p^{2s+2t}}\Bigr).
\end{equation}
As a consequence, $F_\Lambda(s,t;K,\alpha)$ converges absolutely in the
region:
\begin{equation*}
  \{\Re(2s)>1,\Re(s+t)>1,\Re(2t)>1\}=
\{\Re(s)>{\textstyle\frac{1}{2}},\Re(t)>{\textstyle\frac{1}{2}}\}.
\end{equation*}
The proposition now follows immediately from Lemma \ref{lemLQpole},
which implies that $L_{\Lambda,\Q}(s;K,\alpha)$ has a pole at $s=1$
exactly when $K(\sqrt{\alpha})$ is a quadratic or biquadratic field and that
$L_{\Lambda,K}(t;\alpha)$ has a pole at $t=1$ exactly when
$K(\sqrt{\alpha})= K$.
\end{proof}

The proof of Theorem \ref{proppole2} has the following immediate
consequence by simply setting $s=t$.

\begin{cor}\label{polesofZ}
The function $\cD_\Lambda(s,s;K,\alpha)$ has a meromorphic
continuation to the region $\{\Re(s)>\frac12\}$, with at most a pole of
order at most $2$ at $s=1$. More precisely:
\begin{itemize}
\item[{\rm (a)}] $\cD_\Lambda(s,s;K,\alpha)$ has a pole of order $2$
  when $K(\sqrt{\alpha})$ is equal to $K$.
\item[{\rm (b)}] $\cD_\Lambda(s,s;K,\alpha)$ has a pole of order $1$
  when $K(\sqrt{\alpha})$ is a $V_4$-field.
\item[{\rm (c)}] $\cD_\Lambda(s,s;K,\alpha)$ is holomorphic at $s=1$
  when $K(\sqrt{\alpha})$ is a $C_4$- or $D_4$-field.
\end{itemize}
\end{cor}

\subsection{Smoothed sums of Dirichlet coefficients}

Recall that the Mellin transform $\widetilde{\phi}$ of a smooth compactly supported
function of $\phi:\R_{\geq 0}\to\R_{\geq 0}$ is
defined by
\begin{equation*}
\widetilde{\phi}(s):=\int_{t>0}x^s\phi(x)d^\times x,
\end{equation*}
where $d^\times x$ denotes $\frac{dx}{x}$. Standard properties of the
Mellin transform imply that $\widetilde{\phi}$ is entire except for a
possible simple pole at $0$.  Moreover, $\widetilde{\phi}$ has
superpolynomial decay on vertical lines. Mellin inversion then states
that the following equality holds:
\begin{equation*}
\phi(x)=\frac{1}{2\pi i}\int_{\Re(s)=2}x^{-s}\widetilde{\phi}(s)ds.
\end{equation*}

For a Dirichlet series $D(s) = \sum_n a_nn^{-s}$ that is meromorphic
in $\Re(s)>\frac12$ with at most a pole at $s=1$ of order at most $2$,
and with at most polynomial growth on vertical strips, we have
\begin{equation*}
\begin{array}{rcl}
\displaystyle\sum_{n\geq 1}a_{n}\cdot\phi\Bigl(\frac{n}{x}\Bigr)&=&
\displaystyle\frac{1}{2\pi i}\int_{\Re(s)=2}
\sum_{n\geq 1}a_{n}\Bigl(\frac{n}{x}\Bigr)^{-s}
\widetilde{\phi}(s)ds
\\[.2in]&=&
\displaystyle\frac{1}{2\pi i}\int_{\Re(s)=2}
D(s)x^s\widetilde{\phi}(s)ds.
\end{array}
\end{equation*}
Define the ``error'' $E(1/2+\epsilon,D)$ by
\begin{equation}\label{eqerror}
E(1/2+\epsilon,D):=\max_{|t|\ll x^\epsilon}
\{|D\bigl({\textstyle\frac12}+\epsilon+it\bigr)|\}
\end{equation}
Pulling back the line of integration to $\Re(s)=\frac12+\epsilon$, we
pick up the main term from the (possible) double pole of $D(s)$ at
$s=1$, obtaining the following result.
\begin{lemma}\label{lemMellin}
If $\phi: \bR_{\geq 0} \rightarrow \bR_{\geq 0}$ is a smooth compactly supported function and $D(s)$ is a Dirichlet series that is meromorphic in $\Re(s) > \frac{1}{2}$ with at most a pole at $s = 1$ of order at most 2 and with at most polynomial growth on verticial strips, then we have
\begin{equation*}
\sum_{n\geq 1}a_{n}\cdot\phi\Bigl(\frac{n}{x}\Bigr)=
\bigl(\Res^{(2)}_{s=1}D(s)\bigr)X\log X+\bigl(\Res^{(1)}_{s=1}D(s)\bigr)X+
O_\epsilon\bigl(x^{\frac12+\epsilon}
E(1/2+\epsilon,D)\bigr),
\end{equation*}
where $E(1/2+\epsilon, D)$ is given in \eqref{eqerror}.
\end{lemma}

Next, we assume that $D(s)$ is analytic at $s=1$ and estimate its
special value at $1$.
\begin{lemma}\label{lemspv}
  Let $D(s)=\sum_{n>0}a_nn^{-s}$ be holomorphic in $\Re(s)>\frac12$ with at most
  polynomial growth in vertical strips. Pick a smooth compactly
  supported function $\phi:\R_{\geq 0}\to\R_{\geq 0}$ such that
  $\widetilde{\phi}$ has a simple pole at $0$ with residue $1$. Then
  for every $x>0$, we have
  \begin{equation*}
      D(1) =\sum_{n>0}\frac{a_n}{n}\phi\Bigl(\frac{n}{x}\Bigr)+
O_\epsilon\bigl(x^{-\frac12+\epsilon}
E(1/2+\epsilon,D)\bigr),
  \end{equation*}
where $E(1/2+\epsilon,D)$ is given in \eqref{eqerror}.
\end{lemma}
\begin{proof}
  For $x>0$ and
$0<\kappa<\frac12$, we have
\begin{equation}\label{eqMT}
\begin{array}{rcl}
\displaystyle\sum_{n>0}\frac{a_n}{n}\phi\Bigl(\frac{n}{x}\Bigr)&=&
\displaystyle\frac{1}{2\pi i}\int_{\Re(s)=2}\sum_{n>0}\frac{a_n}{n}\Bigl(\frac{n}{x}\Bigr)^{-s}\widetilde{\phi}(s)ds
\\[.2in]&=&
\displaystyle\frac{1}{2\pi i}\int_{\Re(s)=2}D(s+1)x^s\widetilde{\phi}(s)ds
\\[.2in]&=&
D(1)+\displaystyle\frac{1}{2\pi i}\int_{\Re(s)=\frac12+\kappa}D(s)x^{s-1}\widetilde{\phi}(s)ds.
\end{array}
\end{equation}
The result now follows by picking $\kappa=\epsilon$.
\end{proof}

\subsection{Estimates on smooth sums of Dirichlet coefficients}

We begin with introducing notation for the Dirichlet coefficients of
$\cD_\Lambda(s,t;K,\alpha)$.
\begin{defn}\label{anK} Let $\Lambda$ be an acceptable collection of local 
specifications with finite total ramification as in Definition \ref{localspec}, and let $K \in
\cF_2(\Lambda)$ be a quadratic field. Let $K(\sqrt{\alpha}) \in
\cF_4^{\un}(\Lambda,K)$ be a quadratic extension of $K$ that is
unramified outside of the places above $\cS_K(\Lambda)$. Recall the Dirichlet functions $\cD_\Lambda(s,t;K,\alpha)$ as in Definition \ref{dirichletfunction}. We will suppress $\Lambda$ in this
notation, and write
$$\cD_\Lambda(s,s;K,\alpha) := \sum_{n \geq 1}
\frac{a_n(K,\alpha)}{n^s};\quad \quad\cD_\Lambda(s,t;K,\alpha) :=
\sum_{m,n > 0} \frac{a_{m,n}(K,\alpha)}{m^{s}n^{t}}.
$$ When $\alpha$ is trivial, we denote $a_n(K,\alpha)$ and
$a_{m,n}(K,\alpha)$ by $a_n(K)$ and $a_{m,n}(K)$, respectively.
\end{defn}

We have the following bounds on central values of
$\cD_\Lambda(s,s;K,K_\alpha)$, which follow immediately from the
conductor bounds in Lemma \ref{lemLQpole} together with general
subconvexity bounds on $\GL(1)$-automorphic $L$-functions in
\cite{MR2653249}.

\begin{lemma}\label{lemcent}
Let $\Lambda$ be an acceptable collection of local 
specifications with finite total ramification as in Definition \ref{localspec}, and let $K \in
\cF_2(\Lambda)$ be a quadratic field. Let $K(\sqrt{\alpha}) \in
\cF_4^{\un}(\Lambda,K)$ be a quadratic extension of $K$ that is
unramified outside of the places above $\cS_K(\Lambda)$.  We let $P(\Lambda)$ denote square
of the product of the finite primes in $S(\Lambda)$. There exists an absolute constant $\delta>0$ such that
for $s$ with $\Re(s)=\frac12+\epsilon$, we have
\begin{itemize}
\item[{\rm (a)}] $\cD_\Lambda(s,s;K,\alpha)\ll_\epsilon
  |\Disc(K)P(\Lambda) \Im(s)|^{\frac14-\delta}$ when
  $K(\sqrt{\alpha})$ is a quadratic or biquadratic field.
\item[{\rm (b)}] $\cD_\Lambda(s,s;K,\alpha)\ll_\epsilon
  |\Disc(K)P(\Lambda) \Im(s)|^{\frac12-\delta}$ when
  $K(\sqrt{\alpha})$ is a $C_4$- or $D_4$-field.
\end{itemize}
\end{lemma}

We can use the subconvexity estimates stated in the previous lemma and apply Lemma \ref{lemMellin} to pick up the main term from the residues of $D_{\Lambda}(s,s;K,\alpha)$, along with an error bound. In particular, we have the following result.

\begin{lemma}\label{residues}
Let the notation be as in Lemma \ref{lemcent}, and fix a smooth probability function $\phi:\R_{\geq 0}\to\R_{\geq 0}$ with
compact support. For real numbers $X>1$, we have
\begin{equation*}
\displaystyle\sum_{n\geq 1}a_{n}(K)\cdot  \phi\Bigl(\frac{n}{X}\Bigr)
= \Res^{(2)}_{s=1} \cD_\Lambda(s,s;K)\cdot X \log X+\Res^{(1)}_{s=1}
\cD_\Lambda(s,s;K)\cdot X +\,
O_\epsilon\bigl(X^{\frac12+\epsilon}|\Delta(K)P(\Lambda)|^{\frac14-\delta}).
\end{equation*}
For nontrivial $\alpha$, we have:
\begin{equation*}
\sum_{n\geq 1}a_{n}(K,\alpha)\cdot\phi\Bigl(\frac{n}{X}\Bigr)\,\,=\,\, \begin{cases}
\Res^{(1)}_{s=1}\cD_\Lambda(s,s;K,\alpha)\cdot X+
O_\epsilon\bigl(X^{\frac12+\epsilon}|\Disc(K)P(\Lambda)|^{\frac14-\delta}) &
\mbox{ if $K(\sqrt{\alpha})$ is (bi)quadratic;} \vspace{.1in}\\
 O_\epsilon
\bigl(X^{\frac12+\epsilon}|\Delta(K)P(\Lambda)|^{\frac12-\delta}) & \mbox{ otherwise.}
\end{cases}
\end{equation*}
\end{lemma}
\begin{proof}
The proof of the lemma follows immediately from Lemma \ref{lemMellin}
together with the bounds on the central values in Lemma \ref{lemcent}.
\end{proof}

Next, we estimate smoothed double sums of the coefficients of
$\cD_\Lambda(s,t;K,\alpha)$. 

\begin{prop}\label{propddsasymp}
Let $\Lambda$ be an acceptable collection of local 
specifications with finite total ramification as in Definition \ref{localspec}, and let $K \in
\cF_2(\Lambda)$ be a quadratic field. Let $K(\sqrt{\alpha}) \in
\cF_4^{\un}(\Lambda,K)$ be a quadratic extension of $K$ that is
unramified outside of the places above $\cS_K(\Lambda)$. For positive
real numbers $X_{\sigma^2}$ and $X_\tau$ and a smooth probability function $\phi: \bR_{\geq 0} \rightarrow \bR_{\geq 0}$ with compact support, we define the sum
\begin{equation}\label{twovardirichlet}
\cN_{\Lambda,\phi}(X_{\sigma^2},X_\tau;K,\alpha):=
\displaystyle\sum_{m,n>0}a_{m,n}(K,\alpha)
\phi\Bigl(\frac{m}{X_{\sigma^2}}\Bigr)\phi\Bigl(\frac{n}{X_\tau}\Bigr),
\end{equation}
where $a_{m,n}(K,\alpha)$ are the Dirichlet coefficients of $\cD_\Lambda(s,t; K,\alpha) $ given in Definition \ref{anK}. Denote 
$\cN_{\Lambda,\phi}(X_{\sigma^2},X_\tau;K,1)$ by
$\cN_{\Lambda,\phi}(X_{\sigma^2},X_\tau;K)$ when $\alpha$ is
trivial. If $\delta$ is defined as in Lemma \ref{lemcent}, we then have:
\begin{equation*}
\begin{array}{rcl}
  \cN_{\Lambda,\phi}(X_{\sigma^2},X_\tau;K)
=\Res^{(2)}_{s=1}\cD_\Lambda(s,s;K)\cdot X_{\sigma^2}X_\tau
%\vspace{.1in}\\
+\, O_{\epsilon,\Lambda}\bigl(
X_{\sigma^2}^{\frac12+\epsilon}X_\tau\Res_{t=1}\zeta_K(t)
+X_{\sigma^2}X_\tau^{\frac12+\epsilon}
|\Delta(K)P(\Lambda)|^{\frac14-\delta}\bigr),
\end{array}
\end{equation*}
where $P(\Lambda)$ denotes the square of the product of the finite primes in $S(\Lambda)$. 
For nontrivial $\alpha$, we have 
\begin{equation*}
\cN_{\Lambda,\phi}(X_{\sigma^2}, X_\tau;K,\alpha) = \begin{cases}
O_\epsilon\bigl(X_{\sigma^2}X_\tau^{\frac12+\epsilon}
|\Delta(K)P(\Lambda)|^{\frac14-\delta}\bigr)
&\mbox{ if $K(\sqrt{\alpha})$ is biquadratic;} \vspace{.1in}\\
\displaystyle O_\epsilon(X_{\sigma^2}^{\frac12+\epsilon}X_{\tau}^{\frac12+\epsilon}
|\Delta(K)P(\Lambda)|^{\frac12-2\delta}), &\mbox{ otherwise.}
\end{cases}
\end{equation*}
\end{prop}
\begin{proof}
  
We prove the case in detail when $\alpha=1$, i.e., when $K(\sqrt{\alpha}) =
K$. Recall the definition of $F_\Lambda(s,t;K):=F_\Lambda(s,t;K,1)$
from \eqref{eqFfunc}, and note from Lemma \ref{lemLQpole} that
\begin{equation*} G_\Lambda(s,t;K) :=
  \frac{\cD_\Lambda(s,t;K)}{\zeta(s)\zeta_K(t)} \end{equation*} has
the same convergence properties (i.e., absolute convergence to the
right of $\Re(s)>1/2$ and $\Re(t)>1/2$) as $F_\Lambda(s,t;K)$. Therefore, we
have
\begin{equation}
\begin{array}{rcl}
\displaystyle\cN_{\Lambda,\phi}(X_{\sigma^2},X_\tau;K)
&=&
\displaystyle\frac{1}{-4\pi^2}\int_{\Re(s)=2}\int_{\Re(t)=2}
\sum_{m,n>0}a_{m,n}(K)\Bigl(\frac{m}{X_{\sigma^2}}\Bigr)^{-s}
\Bigl(\frac{n}{X_\tau}\Bigr)^{-t}
\widetilde{\phi}(s)\widetilde{\phi}(t)dsdt
\\[.2in]&=&
\displaystyle\frac{1}{-4\pi^2}\int_{\Re(s)=2}\int_{\Re(t)=2}
\zeta(s)\zeta_K(t)G_\Lambda(s,t;K)
X_{\sigma^2}^s X_\tau^t\widetilde{\phi}(s)\widetilde{\phi}(t)dsdt.
\end{array}
\end{equation}
Pulling back the line of integration over $s$ to
$\Re(s)=\frac12+\epsilon$, we see that the left hand side of the above
equation is equal to
\begin{equation*}
\frac{X_{\sigma^2}}{2\pi i}
\int_{\Re(t)=2}\zeta_K(t)G_\Lambda(1,t; K)X_\tau^t\widetilde{\phi}(t)dt
+O\left(X_{\sigma^2}^{\frac12+\epsilon}\int_{\Re(t)=2}\zeta_K(t)
G_\Lambda({\textstyle\frac12}+\epsilon,t;K)
X_\tau^t\widetilde{\phi}(t)dt\right).
\end{equation*}
To estimate the error term above, we merely pull the line of
integration over $t$ to $\Re(t)=1/2+\epsilon$, picking up the pole at
$t=1$, and obtain an error of
\begin{equation*}
O_\epsilon\Bigl(
X_{\sigma^2}^{\frac12+\epsilon}X_\tau |G_\Lambda(1/2+\epsilon,1,K)|
\Res_{t=1}\zeta_K(t)+
X_{\sigma^2}^{\frac12+\epsilon}X_\tau^{\frac12+\epsilon}
|\Delta(K)P(\Lambda)|^{\frac14-\delta}
\Bigr)  
\end{equation*}
which is sufficiently small since $|G_\Lambda(1/2+\epsilon,1,K)|$ is
absolutely bounded independent of $K$. To estimate the main term, we
pull the line of integration over $t$ to $\Re(t)=\frac12+\epsilon$,
yielding
\begin{equation*} 
\Res^{(1)}_{t=1}\zeta_K(s)G_\Lambda(1,1;K)\cdot
X_{\sigma^2}X_\tau+O_\epsilon(X_{\sigma^2}X_\tau^{\frac12+\epsilon}
|\Delta(K)P(\Lambda)|^{\frac14-\delta}),
\end{equation*}
where the error term is again sufficiently small.
Finally, we have
\begin{equation*}
\begin{array}{rcl}
\displaystyle\Res^{(1)}_{t=1}\zeta_K(t)G_\Lambda(1,1;K)&=&
\displaystyle\left.\frac{\zeta(s)\zeta_K(t)G_\Lambda(s,t;K)}{\zeta(s)\zeta(t)}\right|_{s=1,t=1}
\\[.2in]&=&
\displaystyle\left.\frac{\zeta_K(s)\zeta(s)G_\Lambda(s,s;K,K)}{\zeta(s)^2}\right|_{s=1}
\\[.2in]&=&
\displaystyle\Res^{(2)}_{s=1}\cD_\Lambda(s,s;K),
\end{array}
\end{equation*}
as necessary.  The proof for nontrivial $\alpha$ follows the same
lines: when $K(\sqrt{\alpha})$ is biquadratic, we shift the line of
integration from $t=2$ to $t=1/2+\epsilon$ without picking up any
polar contributions, and then shift the line of integration from $s=2$
to $s=1/2+\epsilon$ picking up the pole at $1$. When $K(\sqrt{\alpha})$ has
Galois group $D_4$ or $C_4$, we shift the lines of integration from
$s=2,t=2$ to $s=1/2+\epsilon,t=1/2+\epsilon$ without picking up any
polar contributions. The bounds then follow from the conductor values
listed in Proposition \ref{lemLQpole} and the subconvexity estimates of Lemma
\ref{lemcent}.
\end{proof}

\section{The number of $D_4$-pairs with bounded height}

Let $\Lambda$ be an acceptable collection of local
specifications with finite total ramification as in Definition \ref{localspec}. Recall that $S(\Lambda)$ consists of the infinite place of $\Q$, the even place $2$, all the places at which $\Lambda_p$ contains a  total ramified pair of local algebras, and all places $p$ for which $\Lambda_p$ does not consist of all of the non totally ramified pairs of algebras. Recall also from Definition \ref{sheight} the $S(\Lambda)$-height $\H$ on $\FF_4(\Lambda)$ given by
\begin{equation*}
\H(L,K)=\q^{S(\Lambda)}_{\sigma^2}(L,K)\q^{S(\Lambda)}_{\tau}(L,K)\D^{S(\Lambda)}(L,K),
\end{equation*}
utilizing the fact that the factor corresponding to $\q_\sigma(L,K)$ is $1$ for every $(L,K) \in \FF_4(\Lambda)$ (see Definition \ref{localspec}). We define the
{\it restricted $\Lambda$-mass} $\Mass^\flat(\Lambda)$ to be
\begin{equation}\label{eqmass}
\Mass^\flat(\Lambda):=\Bigl(\sum_{(L,K)\in\Lambda_\infty}\frac{1}{\#\Aut(L,K)}\Bigr)
\prod_{p\in S(\Lambda)} \Bigl(\sum_{(L,K)\in\Lambda_p}
\frac{1}{\#\Aut(L,K)}\Bigr)\Bigl( 1-\frac{1}{p}\Bigr)^3.
\end{equation}
In this section, we prove the following result counting $D_4$-pairs
when they are ordered by height. 
\begin{thm}\label{thD4quar}
Let $\Lambda$ be an acceptable collection of local
specifications with finite total ramification as in Definition \ref{localspec},   and recall $\Mass^\flat(\Lambda)$ defined in \eqref{eqmass}. Let $N_{D_4}(\Lambda;X)$
denote the number of $D_4$-pairs $(L,K)\in\FF_4(\Lambda)$ such that
$\H(L,K)<X$ (see Definition \ref{sheight}). We then have:
\begin{equation*}
N_{D_4}(\Lambda;X)=\frac{1}{4}\Mass^\flat(\Lambda)\prod_{p\not\in S(\Lambda)}
\Bigl(1+\frac{3}{p}\Bigr)\Bigl(1-\frac{1}{p}\Bigr)^3\cdot X\log^2 X
+o(X\log^2 X).
\end{equation*}
\end{thm}

This section is organized as follows. We first consider the ``cuspidal
region'' of the invariant space
$\{(\q_{\sigma^2},\q_{\tau},\D):\q_{\sigma^2}\q_{\tau}\D<X\}$
consisting of triples with at least one small coordinate (bounded by
a power of $\log X$). In \S5.1, we bound the number of $D_4$-pairs in
$\FF_4(\Lambda)$ whose invariants lie in the cuspidal region. Next in
\S5.2, we determine asymptotics for the number of pairs
$(L,K)\in\FF_4(\Lambda)$, whose invariants lie in dyadic ranges inside
the ``main body,'' i.e.~where $\q_{\sigma^2}$, $\q_{\tau}$, $\D \gg (\log X)^A$ for some positive integer $A$. Utilizing Proposition \ref{propddsasymp}, these asymptotics are given in terms of this dyadic
volume times a sum of residues of $\cD(s,s,K)$, as $K$ varies over
$\FF_2(\Lambda)$. Finally, in \S5.3, we evaluate this sum of residues
and then prove the Theorem \ref{thD4quar} using a tiling argument.

\subsection{Bounding the number of $D_4$-pairs in the cusp}

In this section, we work in a slightly more general setting. Let
$\Lambda$ be an acceptable collection of local specifications
with finite total ramification as in Definition \ref{localspec}, and let $T$ be a positive squarefree
integer such that $T$ is not divisible by any prime in
$S(\Lambda)$. Consider the ``$T$-altered'' collection $\Lambda(T)$ of local
specifications defined as follows: for $p\mid T$, set
$\Lambda(T)_p$ to be the set of all totally ramified pairs of local algebras (as in Definition \ref{localspec}),
and for all other places $v$, set $\Lambda(T)_v=\Lambda_v$. Clearly
$\Lambda(T)$ is also acceptable with finite total ramification. In
this subsection, we bound the number of $D_4$ pairs in $\FF_4(\Lambda(T))$ with bounded height and invariants in the ``cuspidal region'', and keep track of the $T$-dependence of the bound. %

\begin{prop}\label{propcusp}
Let $\Lambda$, $T$, and $\Lambda(T)$ be as above. For a positive integer $A$, let $N_{D_4}^{\cusp}(\Lambda(T);X,A)$
denote the number of $D_4$-pairs $(L,K)\in\FF_4(\Lambda(T))$ with
$S(\Lambda(T))$-height less than $X$, and such that one of the following
is satisfied:
\begin{itemize}
\item $\D^{S(\Lambda(T))}(L,K)\leq (\log X)^A$;
\item $\q_\tau^{S(\Lambda(T))}(L,K)\leq (\log X)^A$; or
\item $\q_{\sigma^2}^{S(\Lambda(T))}(L,K)\leq (\log X)^A$.
\end{itemize}
We have the following bound:
\begin{equation*}
N_{D_4}^{\cusp}(\Lambda(T);X,A)\ll_{\Lambda,A,\epsilon}
T^\epsilon X(\log X)^{1+\epsilon}.
\end{equation*}
\end{prop}
\noindent Before we prove the above proposition, we note that we expect (and will later prove!) the growth rate $N_{D_4}(\Lambda;X)\asymp_\Lambda X(\log X)^2$. We expect the same growth rate (with perhaps some factors of $\log T$) for $N_{D_4}(\Lambda(T);X)$, since the $S(\Lambda(T))$-height of pairs of fields $(L,K) \in \FF_4(\Lambda(T))$ ignores primes dividing $T$. Thus the above proposition will imply that the cuspical contribution to $N_{D_4}(\Lambda(T);X)$ is negligible.
\begin{proof}
Note from Remark \ref{remTR} that $B_{\Lambda(T)}(K,\alpha)=0$ unless
$K(\sqrt{\alpha})\in\FF_4^\un(\Lambda(T),K)$. Thus, to analyze
$\Phi_{\Lambda(T)}(s;K)$ and $\Phi_{\Lambda(T)}(s,t;K)$, it is enough
to sum over those $\alpha$ such that $K(\sqrt{\alpha})$ is unramified
outside the places above $\Lambda(T)$. Moreover, the size of
$B_{\Lambda(T)}(K,\alpha)$ is independent of $T$.

We begin by bounding the number of pairs $(L,K)\in\FF_4(\Lambda(T))$
with $S(\Lambda(T))$-height less than $X$ and $|\Delta(K)|/T\ll_\Lambda
\D^{S(\Lambda(T))}(L,K)\leq (\log X)^A$.  For $i=1,2$ and $K(\sqrt{\alpha})\in
\FF_4^\un(\Lambda(T),K)$ we have:
\begin{equation}\label{eqtrresbound}
  \Res_{s=1}^{(i)} \cD_{\Lambda(T)}(s,s;K) \mbox{ and } \Res_{s=1} \cD_{\Lambda(T)}(s,s;K,\alpha)
  \ll_{\Lambda,\epsilon} |\Delta(K)|^\epsilon\ll_{\Lambda,A,\epsilon} (\log X)^\epsilon T^\epsilon,
\end{equation}
by \eqref{eqFfunc}, \eqref{eqFfunc1}, and Brauer--Siegel.
Theorem \ref{thmsec2main} expresses $\Phi_{\Lambda(T)}(K,s)$ in terms
of a weighted sum over $K(\sqrt{\alpha})\in\FF_4^\un(\Lambda(T),K)$ of
$\cD_{\Lambda(T)}(s,s;K,\alpha)$, where these weights are bounded
independently of $T$. Therefore, the number of pairs
$(L,K)\in\FF_4(\Lambda(T))$ with height less than $X$ and
$|\Delta(K)|=\D(L,K)\leq (\log X)^A$ is
\begin{equation*}
\ll\sum_{\substack{|\Delta(K)^{S(\Lambda(T))}|\leq (\log X)^A\\\text{and }\, T\mid\Delta(K)}}
\left(a_n(K)\phi\Bigl(\frac{n|\Delta(K)^{S(\Lambda(T))}|}{X}\Bigr)
+\sum_{K(\sqrt{\alpha})\in\FF_4^\un(\Lambda,K)}a_n(K,\alpha)\phi\Bigl(
\frac{n|\Delta(K)^{S(\Lambda(T))}|}{X}\Bigr)\right),
\end{equation*}
where $\phi:\R_{\geq 0}\to\R_{\geq 0}$ is a compactly supported smooth
function with $\phi(t)=1$ for $t\in[0,1]$, and $a_n(K)$ \& $a_n(K,\alpha)$ are as in Definition \ref{anK}. The above quantity is then
estimated using Lemma \ref{residues} and \eqref{eqtrresbound} to be
\begin{equation*}
\begin{array}{rcl}
&\ll_{\Lambda,A,\epsilon}&\displaystyle
\sum_{\substack{|\Delta(K)|^{S(\Lambda(T))}\leq (\log X)^A\\T\mid\Delta(K)}}
(\log X)^\epsilon T^\epsilon \frac{X}{|\Delta(K)|^{S(\Lambda(T))}}(\log X)
\\[.2in]&\ll_{\Lambda,A,\epsilon}&\displaystyle
T^\epsilon X(\log X)^{1+\epsilon},
\end{array}
\end{equation*}
as necessary.

The required bound on the number of pairs $(L,K)\in\FF_4(\Lambda(T))$
with height less than $X$ and $\q_\tau^{S(\Lambda(T))}(L,K)\leq (\log
X)^A$ follows immediately from the above since the flip
$(L,K)\mapsto(\phi_4(L,K),\phi_2(L,K))$ interchanges the invariants
$\D^{S(\Lambda(T))}$ and $\q_\tau^{S(\Lambda(T))}$.

Finally, we consider the set of pairs $(L,K)\in\FF_4(\Lambda(T))$ with
height less than $X$ and $\q_{\sigma^2}^{S(\Lambda(T))}(L,K)\leq (\log
X)^A$. We first note that it is enough to prove the bound
$O_\epsilon(T^\epsilon X(\log X)^{1+\epsilon})$ for pairs $(L,K)$
whose heights are between $X$ and $2X$. Indeed, the required bound for
pairs with height less than $X$ would follow by simply dividing
$[1,X]$ into dyadic ranges $[X/2^k,X/2^{k+1}]$ and summing the bounds
$1/2^kT^\epsilon X(\log X)^{1+\epsilon}$ over $k$. We fiber over possible
values $q\leq (\log X)^A$ of the invariant $\q_{\sigma^2}$. For each
fixed $q$, we also fix a dyadic range $[Y,2Y]$ for values of
$\D^{S(\Lambda(T))}$ to lie in. Having fixed $q$ and $Y$, it follows
that the various invariants of the pairs $(L,K)$ satisfy these
estimates:
\begin{equation}\label{eqinvrange}
\D(L,K)\asymp_\Lambda TY,\quad \q_\tau(L,K)\asymp_\Lambda \frac{X}{qY},\quad
\q(L,K)=\q_\sigma(L,K)\q_\tau(L,K)q_{\sigma^2}^2(L,K)\asymp_\Lambda\frac{qXT}{Y},
\end{equation}
where the bound on $\q_\tau(L,K)$ follows from the bound on the $S(\Lambda(T))$-height $H = H^{S(\Lambda(T))}$.
As before, by instead counting the flip of the $D_4$-pairs $(L,K)$ if
necessary, we may assume that $TY\leq qXT/Y$. Then \cite[Proposition
  8.2]{D4preprint} implies that the number of $D_4$-pairs in
$\FF_4(\Lambda(T))$ with invariants as in \eqref{eqinvrange} is
bounded by $O_\epsilon(T^\epsilon X/q^{1-\epsilon})$. Adding up over
dyadic ranges $[Y,2Y]$ gives an additional factor of $\log X$, and
adding up over $q\leq (\log X)^A$ gives an additional factor of $(\log
X)^{\epsilon A}$. This gives a total bound of
$O_{\Lambda,A,\epsilon}(T^\epsilon X(\log X)^{1+\epsilon})$, as
necessary.
\end{proof}

\subsection{Estimating the number of $D_4$-pairs in the main body}

Let $\Lambda$ and $T$ be as in the previous subsection, and let $Y$,
$X_{\sigma^2}$, and $X_\tau$ be positive real numbers. Let
$0<\kappa\leq 1$ be fixed. For a
smooth function $\phi:\R_{\geq 0}\to \R_{\geq 0}$ of compact support,
we define
\begin{equation*}
N_{D_4,\phi}^{(\kappa)}(\Lambda(T);Y,X_{\sigma^2},X_\tau):=
\sum_{\substack{(L,K)\in\FF_{4}(\Lambda(T))\\\D^{S(\Lambda(T))}(K)\in[Y,(1+\kappa)Y]}}
\phi\Bigl(\frac{\q_{\sigma^2}^{S(\Lambda(T))}(L,K)}{X_{\sigma^2}}\Bigr)
\phi\Bigl(\frac{\q_{\tau}^{S(\Lambda(T))}(L,K)}{X_{\tau}}\Bigr).
\end{equation*}
For $K\in \FF_2(\Lambda(T))$, recall from Corollary \ref{eqScounttrip}
and Remark \ref{remTR} that the counting function for quartic pairs
$(L,K)\in\FF_4(\Lambda(T),K)$ is given by
\begin{equation*}
\begin{array}{rcl}
\Phi_{\Lambda(T)}(s,t;K)&=&\displaystyle
\sum_{L\in\FF_4(\Lambda(T),K)}\q_{\sigma^2}^{S(\Lambda)}(L,K)^{-s}
\q_{\tau}^{S(\Lambda)}(L,K)^{-t}\\[.2in]&=&\displaystyle
B_{\Lambda(T)}(K)\cD_{\Lambda(T)}(s,t;K)+\sum_{K(\sqrt{\alpha})\in\FF_4^\un(\Lambda(T),K)}
B_{\Lambda(T)}(K,\alpha)\cD_{\Lambda(T)}(s,t;K,\alpha),
\end{array}
\end{equation*}
where it is enough to sum over $K(\sqrt{\alpha})\in\FF_4^\un(\Lambda(T),K)$, and
where the sizes $|B_{\Lambda(T)}(K)|$ and $|B_{\Lambda(T)}(K,\alpha)|$
are independent of $T$. Therefore, we have
\begin{equation}\label{eqMBcotemp}
\begin{array}{rcl}
\displaystyle N_{D_4,\phi}^{(\kappa)}(\Lambda(T);Y,X_{\sigma^2},X_\tau,)&=&\displaystyle
\sum_{|\Delta^{S(\Lambda(T))}(K)|\in[Y,(1+\kappa)Y]}\Bigl(B_{\Lambda(T)}(K)\cN_{\Lambda(T),\phi}(X_{\sigma^2},X_\tau;K)
\\[.2in]&&\displaystyle
\qquad \qquad \quad+\sum_{K(\sqrt{\alpha})\in\FF_4^\un(\Lambda(T),K)}B_{\Lambda(T)}(K,\alpha)\cN_{\Lambda(T),\phi}(X_{\sigma^2},X_\tau;K,\alpha)\Bigr),
\end{array}
\end{equation}
where $\cN_{\Lambda,\phi}(X_{\sigma^2},X_{\tau};K)$ is defined in \eqref{twovardirichlet}.
We start by bounding the contribution to
$N_{D_4,\phi}^{(\kappa)}(\Lambda(T);Y,X_{\sigma^2},X_\tau)$ from
nontrivial $\alpha$.

\begin{proposition}\label{propntrivalB}
Recall the definition of $\delta$ from Lemma \ref{lemcent}. Let $X = YX_{\sigma^2}X_\tau$, and assume that
  $Y,X_{\sigma^2},X_\tau\gg (\log X)^A$ for some positive integer
  $A$. Assume also that $X_\tau\geq Y$.  Then we have
\begin{equation*}  
\sum_{|\Delta^{S(\Lambda(T))}(K)|\in[Y,(1+\kappa)Y]}
\sum_{K(\sqrt{\alpha})\in\FF_4^\un(\Lambda(T),K)}B_{\Lambda(T)}(K,\alpha)
\cN_{\Lambda(T),\phi}(X_{\sigma^2},X_\tau;K,\alpha)
\ll_{\Lambda,A,\epsilon} \frac{T^{1/2-2\delta+\epsilon}X}{(\log X)^{A/4}}.
\end{equation*}
\end{proposition}
\begin{proof}
From Remark \ref{remTR} and Definition \ref{dirichletfunction}, it follows that
the quantities $\cD_{\Lambda(T)}(s,t;K,\alpha)$ are independent of $T$, and hence
that $\cN_{\Lambda(T),\phi}(X_{\sigma^2},X_\tau;K,\alpha)$ are also
independent of $T$. The result now follows immediately from
Proposition \ref{propddsasymp}, applied to
$\cN_{\Lambda,\phi}(X_{\sigma^2},X_\tau;K,\alpha)$, along with the
following facts: the $B_{\Lambda(T)}(K,\alpha)$ are uniformly bounded
independent of $T$, $X_\tau\geq Y$, $|\Delta(K)|\ll_\Lambda TY$,
$YX_{\sigma^2}X_\tau=X$, and $X_{\sigma^2},X_\tau\geq (\log X)^A$.
\end{proof}

We next estimate
$N_{D_4}^{(\kappa)}(\Lambda(T);Y,X_\tau,X_{\sigma^2})$ from the
contribution of the trivial character.
\begin{thm}\label{thMBcount}
Denote the product $YX_{\sigma^2}X_\tau$ by $X$, and assume that
$Y,X_{\sigma^2},X_\tau\geq (\log X)^A$ and $Y\leq X_\tau$. Then we
have
\begin{equation*}
\displaystyle N_{D_4,\phi}^{(\kappa)}(\Lambda;Y,X_{\sigma^2},X_\tau)=
\displaystyle  \Vol(\phi)^2\Bigl(
\sum_{\substack{K\in\FF_2(\Lambda)\\Y\leq\Delta^{S(\Lambda(T))}(K)<(1+\kappa)Y}}
B_{\Lambda}(K)\Res_{s=1}^{(2)}\cD_{\Lambda}(s,s;K)\Bigr)
X_{\sigma^2}X_\tau +
O_{\Lambda,A,\epsilon}\Bigl(\frac{X}{(\log X)^{\frac{A}{4}}}\Bigr).
\end{equation*}
For $T>1$, we have the uniform bound
\begin{equation*}
N_{D_4}^{(\kappa)}(\Lambda(T);Y,X_{\sigma^2},X_\tau)\ll_{\Lambda,A,\epsilon}
XT^\epsilon+\frac{T^{1/2-\delta+\epsilon}X}{(\log X)^{\frac{A}{4}}}
\end{equation*}
\end{thm}
\begin{proof}
From \eqref{eqMBcotemp} and Proposition \ref{propntrivalB} we have
\begin{equation*}
N_{D_4,\phi}^{(\kappa)}(\Lambda(T);Y,X_{\sigma^2},X_\tau,)=
\sum_{|\Delta^{S(\Lambda(T))}(K)|\in[Y,(1+\kappa)Y]}
B_{\Lambda(T)}(K)\cN_{\Lambda(T),\phi}(X_{\sigma^2},X_\tau;K)+O_{\Lambda,A,\epsilon}
\Bigl(\frac{T^{1/2-2\delta+\epsilon}X}{(\log X)^{A/4}}\Bigr).
\end{equation*}
Estimating $\cN_{\Lambda(T),\phi}(X_{\sigma^2},X_\tau;K)$ using
Proposition \ref{propddsasymp} gives us the required main term, and it
only remains to bound the error. Since the $B_{\Lambda(T)}(K)$ are
bounded independent of $T$, this error term is
\begin{equation*}
  \ll_{\Lambda,A,\epsilon}\frac{T^{1/2-2\delta+\epsilon}X}{(\log X)^{A/4}}+\sum_{|\Delta^{S(\Lambda(T))}(K)|\in[Y,(1+\kappa)Y]}
  (X_{\sigma^2}^{1/2+\epsilon}X_\tau\Res_{s=1}\zeta_K(s)
  +X_{\sigma^2}X_\tau^{1/2+\epsilon}(YT)^{1/4-\delta}.
\end{equation*}
The first summand above is clearly small enough. The sum over $K$ of the second summand $X_{\sigma^2}X_\tau^{1/2+\epsilon}(YT)^{1/4-\delta}$ is also small enough
since $Y\leq X_\tau$.

We now consider the first summand in the sum over $K$, and split into two cases. Let
$\theta>0$ be a small positive constant. First suppose that
$T<Y^\theta$. In this case, arguments identical to those in
\cite[\S3]{MR12642} show that the sum over $K\in\FF_2(\Lambda(T))$ of
$\Res_{s=1}\zeta_K(s)$ is bounded by $O(Y)$, which gives a sufficient
bound on the error term. On the other hand, if $T\geq Y^\theta$, then
we simply bound $\Res_{s=1}\zeta_K(s)$ by $O(T^\epsilon)$, which is
also sufficiently small. The first claim of the theorem follows
immediately, and the second claim follows by choosing $\phi$ to be a
positive smooth function that dominates the characteristic function of
$[1,1+\kappa]$, and noting that the weighted sum of residues of
$\cD_{\Lambda(T)}(s,s,K)=\cD_{\Lambda}(s,s,K)$ is bounded by $O(Y)$ in
the next subsection.
\end{proof}

\subsection{Computing the constant}

In this section, we evaluate the leading constant in the main term of
Theorem \ref{thMBcount}. Analogously to the quantity
$\Mass^\flat(\Lambda)$, we define
\begin{equation*}
\Mass_C^\flat(\Lambda)=\Bigl(\sum_{(L,K)\in\Lambda_\infty}\frac{1}{\#\Aut(L,K)}\Bigr)
\prod_{p\in S(\Lambda)} \Bigl(\sum_{(L,K)\in\Lambda_p}
\frac{1}{\#\Aut(L,K)}\Bigr)\Bigl( 1-\frac{1}{p}\Bigr)^2.
\end{equation*}
For a positive real number $I>\max\{p\in S(\Lambda)\}$, we define the function
\begin{equation*}
  D_\Lambda^{(I)}(s;K):=
  \prod_{p\leq I}\cD_{\Lambda,p}(s,s;K)\cdot\prod_{p>I}\cD_{\Lambda,p}(s,2s;K).
\end{equation*}
Then the function $D^{(I)}_\Lambda(s;K)$ has a single pole at $s=1$.  We begin
with a result computing the average residues of $D^{(I)}_\Lambda(s;K)$ over
$K\in \FF_2(\Lambda)$.
\begin{lemma}\label{lemauxres}
Let $Y\to\infty$, we have
\begin{equation*}
  \sum_{\substack{K\in\FF_2(\Lambda)\\Y\leq \Disc^{S(\Lambda)}(K)<\kappa_1Y}}
  \frac{B_\Lambda(K)\Res^{(1)}_{s=1}\cD^{(I)}_\Lambda(s;K)}{\kappa Y}\to
\frac12\Mass^\flat_C(\Lambda)
\prod_{{\scriptstyle p\leq I \atop \scriptstyle p\not\in S(\Lambda)}}\Bigl(1+\frac{3}{p}\Bigr)\Bigl(1-\frac{1}{p}\Bigr)^2
\prod_{p>I}\Bigl(1-\frac{1}{p^2}\Bigr)^2.
\end{equation*}  
\end{lemma}
\begin{proof}
The function $\cD^{(I)}_\Lambda(s;K)$ is the counting function for the set of
quadratic extensions $L$ of $K$ in $\FF_4(\Lambda)$, ordered by $\H_I$, where for such a pair $(L,K)$, we have
\begin{equation*}
  \H_I(L,K)\, :=\, \q_{\tau}^{S(\Lambda)}(L,K)
    {\q_{\sigma^2}^{S(\Lambda)}(L,K)}
    {\q_{\sigma^2}^{S_I\cup S(\Lambda)}(L,K)} \, =
    \, \, \frac{\q_{\tau}^{S(\Lambda)}(L,K)
      \cdot \q_{\sigma^2}^{S(\Lambda)}(L,K)^2}{P_{I,\sigma^2}}
\end{equation*}
where $S_I$ denotes the set of primes $p$ with $p\leq I$ and $p\not\in
S(\Lambda)$, and where $P_{I,\sigma^2} = \displaystyle\prod_{\substack{p\mid
    S_I\\I_p \sim \langle\sigma^2\rangle}}p$. Applying the counting
techniques of the previous two sections to the function
$\cD_\Lambda^{(I)}(s;K)$ yields
\begin{equation}\label{slide}
\begin{array}{rcl}
&&\displaystyle (\kappa Q)\cdot
\sum_{\substack{K\in\FF_2(\Lambda)\\Y\leq \Disc^{S(\Lambda)}(K)<(1+\kappa)Y}}
\Res^{(1)}_{s=1}\cD_\Lambda^{(I)}(s;K)
\\[.2in]&\sim &
\displaystyle\sum_{n\mid P_I} \#\Biggr\{(L,K) \in \cF_4(\Lambda):
\begin{array}{lcl} 
\multicolumn{3}{l}{ Y \leq |\Disc^{S(\Lambda)}(K)| < (1+\kappa)Y,} \vspace{1pt} \\
\multicolumn{3}{l}{ nQ \leq \q_{\tau}^{S(\Lambda)}(L,K)\cdot\q_{\sigma^2}^{S(\Lambda)}(L,K)^2<(1 + \kappa)nQ,}\vspace{1pt}\\
 p\mid \, n&\Rightarrow& I_p(L)\sim \langle \sigma^2 \rangle,\text{ and }\\
 p\mid \frac{P_{I,\sigma^2}}{n}&\Rightarrow&I_p(L)\not\sim\langle\sigma^2\rangle
\end{array}\hspace{-5pt}
\Biggr\},
\end{array}
\end{equation}
for $Y\leq Q$, with $Y$ tending to infinity. We omit the details of
the proof since they are identical to those of the previous
subsection.

For any fixed $n$, asymptotics for the summand corresponding to $n$ in
the right hand side of \eqref{slide} are computed using Theorems 6.9
and 7.9 of \cite{D4preprint} to be
\begin{equation*}
\begin{array}{rl}
\sim&
\displaystyle \kappa^2Y \cdot nQ \cdot
\frac12 \cdot \Mass^\flat_C(\Lambda)
\cdot\prod_{\substack{p\leq I\\p\nmid n \\ p\not\in S(\Lambda)}} \Bigl(1+\frac{2}{p}\Bigr)
\Bigl(1-\frac{1}{p}\Bigr)^2
\cdot\prod_{p>I}\Bigl(1-\frac{1}{p^2}\Bigr)^2\\
=&\displaystyle
\frac{\kappa^2YQ}{2}\cdot \Mass^\flat_C(\Lambda)
\cdot\prod_{p\mid n}\frac{1}{p}\Bigl(1+\frac{2}{p}\Bigr)^{-1}
\cdot\prod_{\substack{p\leq I \\ p\not\in S(\Lambda)}} \Bigl(1+\frac{2}{p}\Bigr)
\Bigl(1-\frac{1}{p}\Bigr)^2
\cdot\prod_{p>I}\Bigl(1-\frac{1}{p^2}\Bigr)^2.
\end{array}
\end{equation*}
Since we have
\begin{equation*}
\sum_{n\mid P_{I,\sigma^2}}\prod_{p\mid n}\frac{1}{p}\Bigl(1+\frac{2}{p}\Bigr)^{-1}
=\prod_{p\leq I,p\not\in S(\Lambda)}\Bigl(1+\frac{1}{p}\bigl(1+\frac{2}{p}\bigr)^{-1}
\Bigr),
\end{equation*}
summing over $n$ and dividing by $\kappa Q$ yields
$$\sum_{\substack{K\in\FF_2(\Lambda)\\\Disc(K)\leq Y}}
\Res^{(1)}_{s=1}\cD_\Lambda^{(I)}(s;K)
  \sim \frac{\kappa Y}{2}\cdot \Mass^\flat_C(\Lambda)
\cdot\prod_{\substack{p\leq I \\ p\not\in S(\Lambda)}} \Bigl(1+\frac{3}{p}\Bigr)
\Bigl(1-\frac{1}{p}\Bigr)^2
\cdot\prod_{p>I}\Bigl(1-\frac{1}{p^2}\Bigr)^2.$$
This concludes the proof of the lemma.
\end{proof}

Next, we compute the weighted average of the double pole residues.

\begin{prop}\label{propConstant}
We have
\begin{equation*}
\lim_{Y\to\infty}\frac{1}{\kappa Y}\sum_{\substack{F\in\FF_2(\Sigma)\\Y<|\Delta(F)|\leq
    (1+\kappa)Y}} B_\Lambda(K) \Res^{(2)}_{s=1}\cD_\Lambda(s,s;K)=
\frac12\Mass^\flat(\Lambda) \prod_{p\not\in
  S(\Lambda)}\Bigl(1+\frac{3}{p}\Bigr)\Bigl(1-\frac{1}{p}\Bigr)^3.
\end{equation*}
\end{prop}
\begin{proof}
First note that we have
\begin{equation*}
\begin{array}{rcl}
\displaystyle\Res_{s=1}^{(2)}\cD_\Lambda(s,s;K)&=&\displaystyle
\frac{\cD_\Lambda(s,s;K)}{\zeta(s)^2}\Bigl|_{s=1}
\, \, =\, \, \displaystyle
\prod_{p}\Bigl(1-\frac{1}{p}\Bigr)^2\cD_{\Lambda,p}(s,s;K)\Bigl|_{s=1}
\\[.2in]&=&\displaystyle
\lim_{I\to\infty}\prod_{p\leq I}\Bigl(1-\frac{1}{p}\Bigr)^2\cD_{\Lambda,p}(s,s;K)
\cdot\prod_{p>I}\Bigl(1-\frac{1}{p}\Bigr)\cD_{\Lambda,p}(s,2s;K)\Bigl|_{s=1}
\\[.2in]&=&\displaystyle
\lim_{I\to\infty}\Bigl[\Res_{s=1}^{(1)}\cD^{(I)}_\Lambda(s;K)\cdot
\prod_{p\leq I}\Bigl(1-\frac{1}{p}\Bigr)\Bigr].
\end{array}
\end{equation*}
The convergence in the above equation is absolute and uniform in $K$
since the $n$th Dirichlet coefficients in all the associated Dirichlet
series are uniformly bounded by $O_\epsilon(n^\epsilon)$, and
supported on powerfull integers $n$. Applying Lemma \ref{lemauxres},
we see that the left hand side of the displayed equation in the lemma
tends to
\begin{equation*}
\frac12\cdot
\Mass^\flat_C(\Lambda)\cdot\prod_{p\in S(\Lambda)}
\Bigl(1-\frac{1}{p}\Bigr)\cdot
\prod_{p\not\in S(\Lambda)}
\Bigl(1+\frac{3}{p}\Bigr)\Bigl(1-\frac{1}{p}\Bigr)^3
=
\frac12\Mass^\flat(\Lambda)\cdot
\prod_{p\not\in S(\Lambda)}
\Bigl(1+\frac{3}{p}\Bigr)\Bigl(1-\frac{1}{p}\Bigr)^3,
\end{equation*}
as necessary.
\end{proof}

We are now ready to prove Theorem \ref{thD4quar}.

\medskip

\noindent {\it Proof of Theorem \ref{thD4quar}.} Let
$Y,X_{\sigma^2},X_\tau$ be positive real constants such that each of
them is $\geq (\log Z)^A$, where $Z=YX_{\sigma^2}X_\tau$ Denote the
number of pairs $(L,K)\in\FF_4(\Lambda(T))$ with
$\D^{S(\Lambda(T))}(L,K)\in[Y,(1+\kappa)Y]$,
$\q_{\sigma^2}^{S(\Lambda(T))}(L,K)\in[X_{\sigma^2},(1+\kappa)X_{\sigma^2}]$,
and $\q_\tau^{S(\Lambda(T))}(L,K)\in[X_\tau,(1+\kappa)X_\tau]$ by
$N_{D_4}^{(\kappa)}(\Lambda(T);Y,X_{\sigma^2},X_\tau)$. Let
$\chi_\kappa$ denote the characteristic function of the interval
$[1,1+\kappa]$, and for every positive integer $n$ choose smooth
functions $\phi_n^+$ and $\phi_n^-$, such that
$\phi_n^+(x)\geq\chi_\kappa(x)\geq\phi_n^-(x)$, and such that
$\Vol(\phi_n^+-\phi_n^-)\to 0$ as $n\to\infty$. When $Y\leq X_\tau$,
Theorem \ref{thMBcount}, applied to this sequence of pairs of smooth
functions implies that we have
\begin{equation*}
\begin{array}{rcl}
\displaystyle N_{D_4}^{(\kappa)}(\Lambda;Y,X_{\sigma^2},X_\tau)
&=&\displaystyle\cN_{\Lambda,\chi_\kappa,\kappa}(D_4;Y,X_{\sigma^2},X_\tau)
\\[.15in]&=&\displaystyle
\frac{\kappa^3}{2}\Mass^\flat(\Lambda)\prod_{p\not\in S(\Lambda)}
\Bigl(1+\frac{3}{p}\Bigr)
\Bigl(1-\frac{1}{p}  \Bigr)^3
\cdot Z+o(Z)+O\bigl(\Vol(\phi_n^+-\phi_n^-)\cdot Z\bigr).
\end{array}
\end{equation*}
Letting $n$ tend to infinity yields
\begin{equation}\label{eqmthfirst}
N_{\Lambda,\kappa}(D_4;Y,X_{\sigma^2},X_\tau)\sim
\frac{\kappa^3}{2}\Mass^\flat(\Lambda)\prod_{p\not\in S(\Lambda)}
\Bigl(1+\frac{3}{p}\Bigr)
\Bigl(1-\frac{1}{p}  \Bigr)^3
\cdot Z.
\end{equation}
Furthermore, we note that \eqref{eqmthfirst} also holds for
$Y>X_\tau$, since the bijection of the flip yields
$$N_{\Lambda,\kappa}(D_4;Y,X_{\sigma^2},X_\tau)=
N_{\Lambda,\kappa}(D_4;X_\tau,X_{\sigma^2},Y).$$

Let $\RR$ be a {\it non-skewed $\kappa$-adic cuboid}, i.e., $\RR$ is
of the form
\begin{equation*}
\bigl\{(x_1,x_2,x_3)\in\R^3: X_i\leq x_i<(1+\kappa)X_i\bigr\},
\end{equation*}
with $X_i\gg (\log (X_1X_2X_3))^A$. Then \eqref{eqmthfirst} asserts
that the number of quartic $D_4$-pairs $(L,K)\in\FF_4(\Lambda)$ whose
invariants satisfy
\begin{equation*}
(|\D^{S(\Lambda)}(L,K)|,\q_{\sigma^2}^{S(\Lambda)}(L,K),\q_\tau^{S(\Lambda)}(L,K))\in\RR
\end{equation*}
is asymptotic to
\begin{equation*}
\frac12 \Mass^\flat(\Lambda)\prod_{p\not\in S(\Lambda)}
\Bigl(1+\frac{3}{p}\Bigr)
\Bigl(1-\frac{1}{p}  \Bigr)^3
\cdot\Vol(\RR),
\end{equation*}
where $\Vol(\RR) = \kappa^3$.
The analogous claim when $\RR$ is the set of triples $(x,y,z)$ whose
product is bounded by $X$ follows by tiling $\RR$ with non-skewed
$\kappa$-adic cuboids and using Proposition \ref{propcusp} to bound
the number of $D_4$-fields with invariants lying in the skewed
regions. We omit the details since they are identical to the proof of
\cite[Theorem 6]{D4preprint}. The theorem then follows from the following
volume computation. We have
\begin{equation*}
\begin{array}{rcl}
\displaystyle\int_{\substack{(x,y,z)\in[1,X]^3\\xyz<X}}dxdydz
&=&\displaystyle X\int_{\substack{(y,z)\in[1,X]^2\\yz<X}}\frac{dydz}{yz}
\\[.2in]&=&\displaystyle
X(\log X)\int_{z=1}^X\frac{dz}{z}-X\int_{z=1}^X\frac{(\log z)dz}{z}
\\[.2in]&=&\displaystyle
\frac{1}{2}X\log^2 X,
\end{array}
\end{equation*}
as necessary. \hfill $\Box$

\section{Proof of the main results}

In this section, we prove Theorems \ref{main} and \ref{congruence
  conditions}. We begin with the following uniformity estimate on the
number of $D_4$-octic fields with large total ramification.
\begin{thm}\label{thunif}
Let $X>Z>0$ be real numbers. The number of octic $D_4$-fields $M$ such
that $\Delta(M)<X$ and the product of primes at which $M$ totally
ramifies is greater than $Z$ is bounded by $$O_\epsilon\left(\frac{X^{1/4}(\log
X)^2}{Z^{\delta-\epsilon}}\right)$$ for some $\delta > 0$.
\end{thm}
\begin{proof}
Let $\Lambda$ denote the (acceptable with finite total ramification)
collection of quartic $D_4$ splitting conditions which prohibits total
ramification at all odd primes. Note that we have
$S(\Lambda)=\{2,\infty\}$. For an odd squarefree positive integer $T$,
the collection $\Lambda(T)$ imposes total ramification at all the
primes dividing $T$, and prohibits total ramification at all other odd
primes. If $M$ is an octic field satisfying the conditions of the
result, then $M$ has a quartic subfield $L$ with quadratic subfield
$K$ satisfying the following conditions:
\begin{itemize}
\item[{\rm (a)}] The pair $(L,K)$ belongs to $\FF_4(\Lambda(T))$ for
  some odd squarefree $T>Z$.
\item[{\rm (b)}] The $S(\Lambda(T))$-height of $(L,K)$ satisfies
  $H(L,K)\ll X^{1/4}/T^{3/2}$.
\item[{\rm (c)}] We have
  $\D^{S(\Lambda(T))}(L,K)<\q_\tau^{S(\Lambda(T))}(L,K)$.
\end{itemize}
From Proposition \ref{propcusp}, it follows that the number of
$D_4$-pairs satisfying the above three conditions, whose invariants lie
in the cuspidal region, is bounded by $O_{A,\epsilon}(X^{1/4}(\log
X)^{1+\epsilon}/T^{3/2-\epsilon})$. Applying Theorem \ref{thMBcount}
for invariants in the main body, and summing over dyadic ranges of
$Y$, $X_{\sigma^2}$, and $X_\tau$ implies that the total number of
$D_4$-pairs satisfying the above three conditions is bounded by
\begin{equation*}
  O_{A,\epsilon}\Bigl(\frac{X^{1/4}(\log X)^{2}}{T^{3/2-\epsilon}}+
  \frac{X^{1/4}}{(\log X)^{\frac{A}{4}-2}T^{1+\delta-\epsilon}}\Bigr).
\end{equation*}
Summing over $T>Z$ yields the result.
\end{proof}

Theorems \ref{thD4quar} and \ref{thunif} allow us to count
quartic-$D_4$ pairs in $\FF_4(\Lambda)$ for acceptable collections
(without the additional condition of finite total ramification.)  For
an acceptable collection $\Lambda$ of quartic local specifications,
let $N_\Delta(\Lambda;X)$ denote the number of pairs
$(L,K)\in\FF_4(\Lambda)$ such that $\Delta(M)<X$, where $M$ denotes
the Galois closure of $L$.
\begin{thm}\label{th:congQuartic}
Let $\Lambda$ be an acceptable collection of quartic local
specifications. Then
\begin{equation*}
N_\Delta(\Lambda;X)=\frac{1}{4}\Bigl(\sum_{(L,K)\in\Lambda_\infty}\frac{1}{\#\Aut(L,K)}\Bigr)
\prod_p\Bigl(\sum_{(L,K)\in\Lambda_p}\frac{|\Delta(M_{(L,K)})|_p^{\frac14}}{\#\Aut(L,K)}\Bigr)
X^{\frac14}\log^2 X^{\frac14}+o(X^{\frac14}\log^2 X),
\end{equation*}
where $M_{(L,K)}$ is the octic algebra corresponding to $(L,K)$.
\end{thm}
\begin{proof}
Since $\Lambda$ is acceptable, it follows that there exists a finite
set $S_\Lambda$ of primes, containing $2$, such that if $p\not\in
S_\Lambda$, then $\Lambda_p$ contains all non-totally ramified pairs.
We assume that $\Lambda$ is a {\it discriminant stable} collection,
i.e., the set $\Lambda_p$ consists of elements whose octic
discriminants have the same valuation $\lambda_p$ for every finite
place $p\in S_\Lambda$. By additivity, the general result clearly
follows from the result on discriminant stable collections, and so
there is no loss of generality in this assumption.  For a positive
integer $T$ not divisible by any primes in $S_\Lambda$, we construct a
collection $\Lambda(T)=(\Lambda(T)_v)_v$ of quartic $D_4$ local
specifications as follows. For $v\in S_\Sigma$ and $p(v) \nmid T$
$\Lambda(T)_v=\Lambda_v$ for all infinite places and for $p \in S_\Sigma$ ; for $p\mid T$, define $\Lambda(T)_p$ to be
the set of all totally ramified extensions; for $p\not\in S_\Sigma$
and $p\nmid T$, define $\Lambda(T)_p$ to be the set of all extensions
which are not totally ramified. Clearly $\Lambda(T)$ is acceptable and
has finite total ramification.

Let $H_T$ denote the $S(\Lambda(T))$-height on $\FF_4(\Lambda(T))$.
Consider a pair $(L,K)\in\FF_4(\Lambda(T))$ and denote the Galois
closure of $L$ over $\bQ$ by $M$. Then we have
\begin{equation*}
|\Delta(M)|=\Bigl(\prod_{p\in S_\Lambda} \lambda_p\Bigr)\cdot T^6\cdot H_T(L,K)^4.
\end{equation*}
Therefore, from Theorem \ref{thD4quar}, we have
\begin{equation*}
\begin{array}{rcl}
\displaystyle N_\Delta(\Lambda(T);X)&=&
\displaystyle N_{D_4}\Bigl(\Lambda;\Bigl(\prod_{p\in
S_\Lambda} \lambda_p^{-1/4} \Bigr)T^{-3/2}X^{1/4}\Bigr)
\\[.2in]&\sim &\displaystyle
\frac14\Mass^\flat(\Lambda(T))\Bigl(\prod_{\substack{p\not\in S_\Lambda\\p\nmid T}}
\Bigl(1+\frac{3}{p}\Bigr)\Bigl(1-\frac{1}{p}\Bigr)^3\Bigr)\cdot
\Bigl(\prod_{p\in S_\Lambda}\lambda_p^{-1/4}\Bigr)T^{-3/2}X^{1/4}\log^2 X^{1/4}.
\end{array}
\end{equation*}
We simplify the above by expanding $\Mass^\flat$ and writing
\begin{equation*}
\Mass^\flat(\Lambda(T))\Bigl(\prod_{p\in S_\Lambda}\lambda_p^{-1/4}\Bigr)T^{-3/2}
=\Bigl(\sum_{(L,K)\in\Lambda_\infty}\frac{1}{\#\Aut(L,K)}\Bigr)
\prod_{p\in S(\Lambda(T))}\Bigl(\sum_{(L,K)\in\Lambda_p}
\frac{|\Delta(M_{(L,K)})|_p^{\frac14}}{\#\Aut(L,K)}\Bigr).
\end{equation*}
The above equality is true since we have
$|\Delta(M_{(L,K)})|_p=\lambda_p^{-1}$ for $p\in S_\Lambda$ and
$(L,K)\in\Lambda_p$, and since for any odd prime $p$, the sum of
$|\Delta(M_{(L,K)})|_p^{\frac14}/\#\Aut(L,K)$ over all totally
ramified pairs is $p^{-3/2}$.

Let $Z>0$ be a real number, and define $\cT(Z)$ to be the set of all
squarefree integers which are the product of primes $p\leq Z$ with
$p\not\in S_\Lambda$. From Theorem \ref{thunif}, we have
\begin{equation*}
N_\Delta(\Lambda;X)=\sum_{T\in\cT(Z)}N_\Delta(\Lambda(T);X)+O_\epsilon
\Bigl(\frac{X^{1/4}\log^2 X}{Z^{\delta-\epsilon}}\Bigr).
\end{equation*}
Since we have
\begin{equation*}
\sum_{(L,K)}\frac{|\Delta(M_{(L,K)})|_p^{\frac14}}{\#\Aut(L,K)}=
\Bigl(1+\frac{3}{p}\Bigr),
\end{equation*}
where the sum is over all non-totally ramified pairs $(L,K)$ over
$\Q_p$, we sum over $T\in\cT(Z)$ and obtain
\begin{equation*}
\begin{array}{rcl}
\displaystyle \frac{N_\Delta(\Lambda;X)}{X^{1/4}\log^2 X^{1/4}}=
\displaystyle\Bigl(\sum_{(L,K)\in\Lambda_\infty}
\frac{1}{\#\Aut(L,K)}\Bigr)
\prod_{p\leq Z}\Bigl(\sum_{(L,K)\in\Lambda_p}
\frac{|\Delta(M_{(L,K)})|_p^{\frac14}}{\#\Aut(L,K)}\Bigr)
\prod_{p>Z}\Bigl(1+\frac{3}{p}\Bigr)\Bigl(1-\frac{1}{p}\Bigr)^3
+O_\epsilon(Z^{-\delta+\epsilon}).
\end{array}
\end{equation*}
Letting $Z$ tend to infinity yields the result.
\end{proof}

Let $v$ be any place of $\Q$, and 
let $\Sigma_v^\all$ denote the set of isomorphism classes of octic
algebras of $D_4$-type over $\Q_v$. Since a $D_4$-type uniquely determines an \'etale octic extension of $\Q_v$, the set $\Sigma_v^\all$ can be canonically identified with the set of $D_4$-types over $\Q_v$. Our next result describes the relation between the set of $D_4$ types over $\Q_v$ and $\Lambda_v^\all$. 
\begin{proposition}\label{prop:LambtoSig}
There is a natural bijection between the set of $D_4$-types over $\Q_v$ and $\Lambda_v^\all$. Moreover, if the $D_4$-type $\rho$ corresponds to $(L,K)\in\Lambda_v^\all$ under the above bijection, then the groups $\Aut([\rho])$ and $\Aut(L,K)$ are isomorphic.
\end{proposition}
\begin{proof}
Consider the set $\mathcal{M}(\Q_v,S_4)$ of conjugacy classes of continuous homomorphisms $G_{\Q_v}\to S_4$, up to conjugacy by elements in $S_4$. Following the treatment in \cite[Lemma 3.1]{MR2354797}, we describe a bijection between $\mathcal{M}(\Q_v,S_4)$ and \'etale quartic algebras of $\Q_v$: given an \'etale quartic algebra $L$ over $\Q_v$, we see that $L\otimes\overline{\Q}_v\cong\overline{\Q}_v^4$ contains $4$ minimal idempotents. Then the action of $G_{\Q_v}$ on $L\otimes\overline{\Q}_v$, given by the trivial action on the first factor and the usual action on the second, fixes this set of minimal idempotents, yielding a continuous homomorphism $G_{\Q_v}\to S_4$. Conversely, 
a continuous homomorphism $G_{\Q_v}\to S_4$ gives an action of $G_{\Q_v}$ on $\overline{\Q}_v^4$. The set of invariants for this action forms a $\Q_v$-subalgebra which is the required \'etale quartic extension of $\Q_v$. Moreover, as explained in the proof of \cite[Lemma 3.2]{MR2354797}, if $L$ is an \'etale quartic algebra over $\Q_v$ corresponding to the representation $\rho:G_{\Q_v}\to S_4$, then the group $\Aut(L)$ is isomorphic with the group of $G_{\Q_v}$-equivariant permutations of the minimal idempotents in $L\otimes\overline{\Q}_v$, which in turn can be identified with the centralizer of $\rm{Im}(\rho)$ in $S_4$.

Now suppose that $\rho$ is a $D_4$-type, i.e., $\rho:G_{\Q_v}\to D_4$ is a continuous homomorphism. We think of $D_4$ as acting on the vertices of a square, and thus fix an embedding $D_4\to S_4$. From the above construction, we obtain an \'etale quartic extension $L$ of $\Q_v$ corresponding to $\rho':G_{\Q_v}\to D_4\to S_4$. Being in the center of $D_4$, the element $\sigma^2$ yields a nontrivial involution $c\in\Aut(L)$. Let $K$ denote the set of elements fixed by $c$, and note that $K$ forms an \'etale quadratic extension of $\Q_v$. Thus the pair $(L,K)$ belongs to $\Lambda_v^\all$. Conversely, given a pair $(L,K)\in\Lambda_v^\all$, let $\rho':G_{\Q_v}\to S_4$ be the continuous homomorphism corresponding to $L$. The involution $c$ of $L$ fixing $K$ acts as a double transposition on the set of four minimal idempotents in $L\otimes\overline{\Q}_v$, which implies that the image of $\rho'$ is contained within the subgroup $D_4\subset S_4$ of permutations of these idempotents which commute with $c$. Hence the pair $(L,K)$ yields a $D_4$-type $\rho:G_{\Q_v}\to D_4$. It is easy to check that these two constructions are inverses of each other, thus yielding the required bijection between the set of $D_4$-types of $\Q_v$ and $\Lambda_v^\all$.

Finally, keeping the above notation, note that $\Aut([\rho])$, which by definition is the centralizer of the image of $\rho$ in $D_4$, is therefore equal to the subgroup of the centralizer of the image of $\rho'$ in $S_4$ consisting of elements which commute with the double transposition corresponding to $c\in\Aut(L)$. Since $c$ generates $\Aut(L/K)$, we see that $\Aut([\rho])$ is isomorphic to the subgroup of authomorphisms of $L$ which fix $K$, as necessary.
\end{proof}

Theorem \ref{congruence conditions} follows immediately from Theorem \ref{th:congQuartic} and Proposition \ref{prop:LambtoSig}. We are now ready to prove Theorem \ref{main}.

\medskip

\noindent {\bf Proof of Theorem \ref{main}:} In light of Theorem \ref{congruence conditions}, all that remains is to compute the local factors
\begin{equation*}
\sum_{\rho_\infty\in\Sigma_\infty^{(i)}}\frac{1}{\Aut(\rho_\infty)},\quad\quad
\sum_{\rho_p\in\Sigma_p^\all}\frac{|\Delta(M_p)|^\frac14_p}{\Aut(\rho_p)},
\end{equation*}
for $i=r$ and $i=c$, and for primes $p$, where $M_p$ is the \'etale octic $\Q_p$-extension corresponding to $\rho$. It is easy to check that when $v=\infty$, the sum is $1/4$ in the totally real case and $1/2$ in the complex case. The computation of the local factors for finite primes follow from the following tables; the first, Table \ref{tab:p}, is for primes $p\neq 2$, while the second, Table \ref{tab:2} is for $p=2$.

\begin{table}
\centering
\begin{tabular}{|c|c|c|c|c|c|c|}
\hline
$\#$ of Alg & $\sigma(L,K)$ & $\Gal(L)$ &  $\Aut(L,K)$ &
$\#\Aut(L,K)$ & $\Delta(M)$ \\
\hline
$1$ & $((11),(1111))$ & -   &  $D_4$  & $8$ & $1$ \\
$1$ & $((11),(112))$  & -   &  $V_4$  & $4$ & $1$ \\
$1$ & $((11),(22))$   & -   &  $D_4$  & $8$ & $1$ \\
$1$ & $((2),(22))$    & -   &  $V_4$  & $4$ & $1$ \\
$1$ & $((2),(4))$     & $C_4$ &  $C_4$  & $4$ & $1$ \\
\hline
$2$ & $((11),(1^211))$   & -   & $V_4$  & $4$ & $p^4$\\
$2$ & $((11),(1^22))$    & -   & $V_4$  & $4$ & $p^4$ \\
$2$ & $((1^2),(1^21^2))$ & -   & $V_4$  & $4$ & $p^4$ \\
$2$ & $((1^2),(2^2))$    & $V_4$   & $V_4$  & $4$ & $p^4$ \\
\hline
$2$ & $((11),(1^21^2))$  & -   & $D_4$  & $8$ & $p^4$\\
$1$ & $((11),(1^21^{2'}))$  & -   & $V_4$  & $4$ & $p^4$\\
$1$ & $((2),(2^2))$      & $V_4$   &  $V_4$  & $4$ & $p^4$ \\
$1$ & $((2),(2^2))$      & $C_4$   &  $C_4$  & $4$ & $p^4$\\
$(0,2)$ & $((1^2),(1^4))$    & $D_4$   &  $C_2$  & $2$ & $p^6$ \\
$(4,0)$ & $((1^2),(1^4))$    & $C_4$   &  $C_4$  & $4$ & $p^6$ \\
\hline
\end{tabular}
\caption{Local invariants for quartic $D_4$ extensions of $\Q_p$. When we write $(a,b)$ in the first column, the number is $a$ for primes
of the form $4k+1$ and $b$ for primes of the form $4k+3$. Also, $(1^21^{2'})$ indicates that the two ramified algebras (whose direct sum make up $K$) are different while $(1^21^2)$ indicates that the two ramified algebras (whose direct sum make up $K$) are the same. 
}\label{tab:p}
\end{table}

The mass formula for
octic discriminants at odd primes $p$
is:
\begin{equation*}
\Bigl(1+\frac{3}{p}+\frac{1}{p^{3/2}}\Bigr)\Bigl(1-\frac{1}{p}\Bigr)^3.
\end{equation*}
\begin{table}
\centering
\begin{tabular}{|c|c|c|c|c|c|c|}
\hline
$\#$ of Alg & $\sigma(L,K)$ & $\Gal(L)$ &  $\Aut(L,K)$ &
$\#\Aut(L,K)$ & $\Delta(M)$ \\
\hline
$1$ & $((11),(1111))$ & -   &  $D_4$  & $8$ & $1$ \\
$1$ & $((11),(112))$  & -   &  $V_4$  & $4$ & $1$ \\
$1$ & $((11),(22))$   & -   &  $D_4$  & $8$ & $1$ \\
$1$ & $((2),(22))$    & -   &  $V_4$  & $4$ & $1$ \\
$1$ & $((2),(4))$     & $C_4$ &  $C_4$  & $4$ & $1$  \\
\hline
$2$ & $((11),(1^211))$   & -   & $V_4$  & $4$ & $p^8$ \\
$4$ & $((11),(1^211))$   & -   & $V_4$  & $4$ & $p^{12}$ \\

$2$ & $((11),(1^22))$    & -   & $V_4$  & $4$ & $p^8$ \\
$4$ & $((11),(1^22))$    & -   & $V_4$  & $4$ & $p^{12}$ \\

$2$ & $((1^2),(1^21^2))$ & -   & $V_4$  & $4$ & $p^8$ \\
$4$ & $((1^2),(1^21^2))$ & -   & $V_4$  & $4$ & $p^{12}$ \\

$2$ & $((1^2),(2^2))$    & $V_4$   & $V_4$  & $4$ & $p^8$ \\
$4$ & $((1^2),(2^2))$    & $V_4$   & $V_4$  & $4$ & $p^{12}$ \\
\hline\hline
$2$ & $((11),(1_2^21_2^2))$  & -   & $D_4$  & $8$ & $p^8$ \\
$4$ & $((11),(1_3^21_3^2))$  & -   & $D_4$  & $8$ & $p^{12}$ \\

$1$ & $((11),(1_2^21_2^{2'}))$  & -   & $V_4$  & $4$ & $p^8$ \\
$2$ & $((11),(1_3^21_3^{2'}))$  & -   & $V_4$  & $4$ & $p^{12}$\\
$4$ & $((11),(1_3^21_3^{2'}))$  & -   & $V_4$  & $4$ & $p^{16}$\\

$8$ & $((11),(1_2^21_3^2))$  & -   & $V_4$  & $4$ & $p^{16}$\\
\hline
$1$ & $((2),(2^2))$      & $V_4$   &  $V_4$  & $4$ & $p^8$\\
$1$ & $((2),(2^2))$      & $C_4$   &  $C_4$  & $4$ & $p^8$\\
$2$ & $((2),(2^2))$      & $D_4$   &  $C_2$  & $2$ & $p^8$\\
$2$ & $((2),(2^2))$      & $V_4$   &  $V_4$  & $4$ & $p^{12}$\\
$2$ & $((2),(2^2))$      & $C_4$   &  $C_4$  & $4$ & $p^{12}$ \\
$2$ & $((2),(2^2))$      & $D_4$   &  $C_2$  & $2$ & $p^{12}$ \\
\hline
$4$ & $((1_2^2),(1^4))$    & $V_4$   &  $V_4$  & $4$ & $p^{16}$\\
$8$ & $((1_3^2),(1^4))$    & $V_4$   &  $V_4$  & $4$ & $p^{16}$\\
$8$ & $((1_3^2),(1^4))$    & $C_4$   &  $C_4$  & $4$ & $p^{22}$\\

$2$ & $((1_2^2),(1^4))$    & $D_4$   &  $C_2$  & $2$ & $p^{12}$\\
$2$ & $((1_2^2),(1^4))$    & $D_4$   &  $C_2$  & $2$ & $p^{16}$\\
$8$ & $((1_2^2),(1^4))$    & $D_4$   &  $C_2$  & $2$ & $p^{22}$\\
$8$ & $((1_3^2),(1^4))$    & $D_4$   &  $C_2$  & $2$ & $p^{22}$\\

$4$ & $((1_3^2),(1^4))$    & $D_4$   &  $C_2$  & $2$ & $p^{22}$\\
$8$ & $((1_3^2),(1^4))$    & $D_4$   &  $C_2$  & $2$ & $p^{24}$\\
\hline
\end{tabular}
\caption{Local invariants for quartic $D_4$ extensions of $\Q_2$. 
Above, $1^2_2$ refers to the case when the ramified quadratic subalgebra $K/\bQ_2$ has discriminant $2^2$. Likewise, $1^2_3$ refers to the case when the ramified quadratic subalgebra $K/\bQ_2$ has discriminant $2^3$. Also, 
$(1^2_3 1^{2'}_3)$ and $(1^2_2 1^{2'}_2)$ 
indicates that the two ramified algebras (whose direct sum make up $K$) are different while $(1_3^21_3^2)$ and $(1_2^21_2^2)$ indicates that the two ramified algebras (whose direct sum make up $K$) are the same.
}\label{tab:2}
\end{table}
The mass formula for octic discriminants at $p = 2$ is:
$$3 + \frac{1}{2} + \frac{3}{2^{7/2}}.$$

\bibliographystyle{abbrv}
\bibliography{references}

\begin{thebibliography}{10}

\bibitem{MR4047213}
B.~Alberts.
\newblock The weak form of {M}alle's conjecture and solvable groups.
\newblock {\em Res. Number Theory}, 6(1):Paper No. 10, 23, 2020.

\bibitem{ALWWPreprint}
B.~Alberts, R.~J. Lemke~Oliver, J.~Wang, and M.~M. Wood.
\newblock Inductive methods for counting number fields.
\newblock {\em arXiv preprint 2501.18574}, 2025.

\bibitem{D4preprint}
S.~A. Altu\u{g}, A.~Shankar, I.~Varma, and K.~H. Wilson.
\newblock The number of {$D_4$}-fields ordered by conductor.
\newblock {\em J. Eur. Math. Soc. (JEMS)}, 23(8):2733--2785, 2021.

\bibitem{MR2183288}
M.~Bhargava.
\newblock The density of discriminants of quartic rings and fields.
\newblock {\em Ann. of Math. (2)}, 162(2):1031--1063, 2005.

\bibitem{MR2354798}
M.~Bhargava.
\newblock Mass formulae for extensions of local fields, and conjectures on the
  density of number field discriminants.
\newblock {\em Int. Math. Res. Not. IMRN}, (17):Art. ID rnm052, 20, 2007.

\bibitem{MR2745272}
M.~Bhargava.
\newblock The density of discriminants of quintic rings and fields.
\newblock {\em Ann. of Math. (2)}, 172(3):1559--1591, 2010.

\bibitem{MR2373587}
M.~Bhargava and M.~M. Wood.
\newblock The density of discriminants of {$S_3$}-sextic number fields.
\newblock {\em Proc. Amer. Math. Soc.}, 136(5):1581--1587, 2008.

\bibitem{MR1918290}
H.~Cohen, F.~Diaz~y Diaz, and M.~Olivier.
\newblock Enumerating quartic dihedral extensions of {$\Bbb Q$}.
\newblock {\em Compositio Math.}, 133(1):65--93, 2002.

\bibitem{MR936994}
B.~Datskovsky and D.~J. Wright.
\newblock Density of discriminants of cubic extensions.
\newblock {\em J. Reine Angew. Math.}, 386:116--138, 1988.

\bibitem{MR491593}
H.~Davenport and H.~Heilbronn.
\newblock On the density of discriminants of cubic fields. {II}.
\newblock {\em Proc. Roy. Soc. London Ser. A}, 322(1551):405--420, 1971.

\bibitem{MR2354797}
K.~S. Kedlaya.
\newblock Mass formulas for local {G}alois representations.
\newblock {\em Int. Math. Res. Not. IMRN}, (17):Art. ID rnm021, 26, 2007.
\newblock With an appendix by Daniel Gulotta.

\bibitem{habilitation}
J.~Kl\"uners.
\newblock {\em \"Uber die Asymptotik von Zahlk\"orpern mit vorgegebener
  Galoisgruppe}.
\newblock Shaker, 2005.

\bibitem{MR2904935}
J.~Kl\"{u}ners.
\newblock The distribution of number fields with wreath products as {G}alois
  groups.
\newblock {\em Int. J. Number Theory}, 8(3):845--858, 2012.

\bibitem{MR2076117}
J.~Kl\"{u}ners and G.~Malle.
\newblock Counting nilpotent {G}alois extensions.
\newblock {\em J. Reine Angew. Math.}, 572:1--26, 2004.

\bibitem{KP}
P.~Koymans and C.~Pagano.
\newblock On {M}alle's conjecture for nilpotent groups.
\newblock {\em Trans. Amer. Math. Soc. Ser. B}, 10:310--354, 2023.

\bibitem{LS}
D.~Loughran and T.~Santens.
\newblock Malle's conjecture and brauer groups of stacks, 2024.

\bibitem{MR791087}
S.~M\"{a}ki.
\newblock On the density of abelian number fields.
\newblock {\em Ann. Acad. Sci. Fenn. Ser. A I Math. Dissertationes}, (54):104,
  1985.

\bibitem{MR1200974}
S.~M\"{a}ki.
\newblock The conductor density of abelian number fields.
\newblock {\em J. London Math. Soc. (2)}, 47(1):18--30, 1993.

\bibitem{MR1884706}
G.~Malle.
\newblock On the distribution of {G}alois groups.
\newblock {\em J. Number Theory}, 92(2):315--329, 2002.

\bibitem{MR2068887}
G.~Malle.
\newblock On the distribution of {G}alois groups. {II}.
\newblock {\em Experiment. Math.}, 13(2):129--135, 2004.

\bibitem{MR2653249}
P.~Michel and A.~Venkatesh.
\newblock The subconvexity problem for {${\rm GL}_2$}.
\newblock {\em Publ. Math. Inst. Hautes \'Etudes Sci.}, (111):171--271, 2010.

\bibitem{STCubicInv}
A.~Shankar and F.~Thorne.
\newblock On the asymptotics of cubic fields ordered by general invariants.
\newblock 2022.
\newblock preprint.

\bibitem{MR12642}
C.~L. Siegel.
\newblock The average measure of quadratic forms with given determinant and
  signature.
\newblock {\em Ann. of Math. (2)}, 45:667--685, 1944.

\bibitem{MR2581243}
M.~M. Wood.
\newblock On the probabilities of local behaviors in abelian field extensions.
\newblock {\em Compos. Math.}, 146(1):102--128, 2010.

\bibitem{MR969545}
D.~J. Wright.
\newblock Distribution of discriminants of abelian extensions.
\newblock {\em Proc. London Math. Soc. (3)}, 58(1):17--50, 1989.

\end{thebibliography}

\bigskip{\footnotesize% \textit{E-mail address},
    \texttt{ashankar@math.toronto.edu and ila@math.toronto.edu}\par \textsc{Department of
      Mathematics, University of Toronto, Toronto, ON, M5S 2E4,
      Canada}
}

\end{document}